%% file: main.tex
\documentclass[11pt]{article}
\usepackage{amsmath,amssymb}
\usepackage{prettyref}
\usepackage{qr}
\renewcommand{\gets}{\leftarrow}
\renewcommand{\bC}{\mathbb{C}}
\renewcommand{\bR}{\mathbb{R}}
\usepackage{boxedminipage}
\usepackage{enumitem}

\newcommand{\muqr}{\nu_{\mathsf{IQR}}}

\newcommand{\croot}{c_{\mathsf{root}}}
\newcommand{\Croot}{C_{\mathsf{root}}}

\newcommand{\cu}{C_{\mathsf{U}}}
\newcommand{\cd}{C_{\mathsf{D}}}
\newcommand{\Cg}{C_{\mathsf{G}}}

\newcommand{\false}{\texttt{false}}
\newcommand{\true}{\texttt{true}}

\newcommand{\deflate}{\mathsf{Deflate}}

\newcommand{\scale}{\Sigma} 
\newcommand{\cor}{\mathsf{correctness}}
\newcommand{\error}{\mathsf{error}}

\newcommand{\chs}{\check{s}}

\newcommand{\iqr}{\mathsf{IQR}}

\newcommand{\dispec}{\mathsf{DistSpec}}
\newcommand{\oneig}{\mathsf{OneEig}}
\newcommand{\comptau}[1]{\mathsf{Tau}^{#1}}

\newcommand{\hform}{\mathsf{HessBU}}

\newcommand{\rhform}{\mathsf{RHess}}
\newcommand{\hous}[1]{\mathsf{Hous}(#1)}

\newcommand{\smalleig}{\mathsf{SmallEig}}
\newcommand{\findritz}{\mathsf{SmallEig}}
\newcommand{\eig}{\mathsf{SmallEig}}

\newcommand{\unif}{\mathrm{Unif}}

\newcommand{\wcc}{\omega}

\newcommand{\tol}{\eta_1}
\newcommand{\pretol}{\eta_2}
\newcommand{\tolr}{\eta_1}
\newcommand{\pretolr}{\eta_2}

\newcommand{\chn}{\check{\mathcal{N}}}

\newcommand{\wfail}{\varphi}

\newcommand{\forward}{\beta}
\newcommand{\fward}{\beta}
\newcommand{\bck}{\delta}

\newcommand{\ch}{c_{\mathsf{H}}}
\newcommand{\chh}{c_{\mathsf{h}}}
\newcommand{\crh}{c_{\mathsf{RH}}}

\newcommand{\Rho}{\mathrm{P}}

\newcommand{\Beta}{\mathrm{Beta}}

\newcommand{\area}{\mathrm{Area}}

\newcommand{\Spec}[1]{\mathrm{Spec}\,{#1}}
\newcommand{\dist}{\mathrm{dist}}

\newcommand{\bN}{\mathbb{N}}
\renewcommand{\bS}{\mathbb{S}}

\newcommand{\dd}{\tau}
\newcommand{\nn}{\ax{\tau^m}}

\newcommand{\rad}{C}

\newcommand{\qr}{\mathrm{qr}}

\newcommand{\shat}{\zeta}

\newcommand{\ds}[2]{\mathrm{dist}(#1, \Spec{#2})}
\newcommand{\rt}{\mathsf{root}}

\newcommand{\hess}{H}

\newcommand{\w}{w}

\newcommand{\E}{\mathbb{E}}
\renewcommand{\P}{\mathbb{P}}
\newcommand{\ax}[1]{\widetilde{#1}}

\renewcommand{\next}[1]{\widehat{#1}}

\newcommand{\fl}{\mathsf{fl}}

\newcommand{\mach}{\textbf{\textup{u}}}
\newcommand{\acc}{\delta}

\newcommand{\wacc}{\omega}

\newcommand{\gap}{\mathrm{gap}}

\newcommand{\absacc}{\Delta}

\renewcommand{\r}{\theta}

\newcommand{\decouple}{{\mathsf{Decouple}}}

\title{Global Convergence of  Hessenberg Shifted QR III: Approximate Ritz Values via Shifted Inverse Iteration}
\author{Jess Banks\thanks{\texttt{jess.m.banks@berkeley.edu}. Supported by NSF GRFP Grant DGE-1752814 and NSF Grant  CCF-2009011.}\\ UC Berkeley \and  Jorge Garza-Vargas\thanks{\texttt{jgarzavargas@berkeley.edu}. Supported by NSF Grant  CCF-2009011.}\\ UC Berkeley \and  Nikhil Srivastava\thanks{\texttt{nikhil@math.berkeley.edu}. Supported by NSF Grant  CCF-2009011.}  \\    UC Berkeley }

\begin{document}
\maketitle

\begin{abstract}
    We give a self-contained randomized algorithm based on shifted inverse iteration which provably computes the eigenvalues of an arbitrary matrix $M\in\bC^{n\times n}$ up to backward error $\delta\|M\|$ in $O(n^4+n^3\log^2(n/\delta)+\log(n/\delta)^2\log\log(n/\delta))$ floating point operations using $O(\log^2(n/\delta))$ bits of precision. While the $O(n^4)$ complexity is prohibitive for large matrices, the algorithm is simple and may be useful for provably computing the eigenvalues of small matrices using controlled precision, in particular for computing Ritz values in shifted QR algorithms as in \cite{banks2022II}.
\end{abstract}


\tableofcontents


\input{intro.tex}

\input{preliminaries}

\input{oneig}

\input{decoupling}

\input{randhess}

\input{eig}

\bibliographystyle{alpha}
\bibliography{Eig}

\appendix

\input{randvector}

\input{shattering}

\end{document}

%% file: intro.tex
\section{Introduction}

In Part I of this series \cite{banks2021global} we gave a family of shifting strategies, of some suitable degree $k$, for which the Hessenberg shifted QR algorithm converges globally and rapidly in exact arithmetic on nonsymmetric matrices with controlled eigenvector condition number. Our analysis relied on the existence of an algorithm, which we called a \emph{Ritz value finder}, which on a matrix input $H$ and accuracy parameter $\r>1$ would compute a set $\calR = \{r_1, \dots, r_k\}$ of $\r$-optimal Ritz values for $H$, i.e. a set of complex numbers satisfying 
\begin{equation*}
    \left\|e_n^* \prod_{i\le k}(H-r_i)\right\|^{1/k}\le \r  \min_{p\in\calP_k} \|e_n^*p(H)\|^{1/k},
\end{equation*}
where $\calP_k$ denotes the set of monic polynomials of degree $k$. Then, in Part II \cite{banks2022II} we showed that any algorithm that could solve the forward-error eigenproblem could be used (on the lower-right $k\times k$ corner of $H$) to build a Ritz value finder.   In this paper (Part III of the series) we complete our analysis by presenting a randomized algorithm $\eig$, based on shifted inverse iteration, that can solve this eigenvalue problem on any input $M\in \bC^{n\times n}$ using a controlled amount of precision in floating point arithmetic.  Our main result can be stated as follows.

\begin{theorem}         
\label{thm:main}
On any input matrix $M\in \bC^{n\times n}$, accuracy parameter $\delta >0$, and failure probability tolerance $\phi>0$, the algorithm $\eig(M, \delta, \phi)$ produces, with probability $1-\phi$, the eigenvalues of a matrix $\ax{M}\in \bC^{n\times n}$ with $\|M-\ax{M}\|\leq \delta \|M\|$, using at most 
$$O\big(n^4+ n^3 \log (n/\delta\phi)^2+\log(n/\delta\phi)^2 \log \log(n/\delta\phi)    \big)$$
arithmetic operations on a floating point machine with $O(\log(n/\delta \phi)^2)$ bits of precision. 
\end{theorem}

The above theorem shows that, when implemented on a floating point machine with $O(\log(n/\delta \phi)^2)$ bits of precision, the algorithm $\eig$ is $\delta$-backward stable. However, as mentioned above, the shifting strategy analyzed in \cite{banks2021global, banks2022II} requires an algorithm that provides forward approximations of the eigenvalues of a (small) matrix. The following result  \cite[Theorem 39.1]{bhatia2007perturbation} turns any backward error algorithm for the eigenproblem into a forward error algorithm, at the cost of multiplying  the number of bits of precision by a roughly $n$ (which might be tolerable for small $n$, but prohibitively expensive otherwise).  

\begin{lemma}
    \label{thm:backward-to-forward-eig}
    Let $M, \ax{M} \in \bC^{n\times n}$ be any two matrices. Then there are labellings $\lambda_1,...,\lambda_n$ and $\ax\lambda_1,...,\ax\lambda_n$ of the eigenvalues of $M$ and $\ax M$, respectively, so that
    $$
        \max_i |\lambda_i - \ax\lambda_i| \le 4(\|M\| + \|\ax M\|)^{1 - 1/n}\|M - \ax M\|^{1/n}.
    $$
\end{lemma}

 In particular, for every $\forward \le 1$ one can produce $\forward$-forward approximate eigenvalues by calling $\smalleig$ with accuracy
$$
    \acc = \left( \frac{\beta}{12} \right)^n,
$$
as $\|\ax M\| \le \|M\| + \delta \le 2\|M\|$. This yields the following corollary. 

\begin{corollary}
On any input matrix $M\in \bC^{n\times n}$ with eigenvalues $\lambda_1, \dots, \lambda_n$,  any accuracy parameter $\beta >0$, and failure probability tolerance $\phi>0$, one can use $\eig$ to find, with probability $1-\phi$, approximate eigenvalues $\ax{\lambda}_1, \dots, \ax \lambda_n\in \bC^{n\times n}$ such that
$$\max_{i} |\lambda_i-\ax{\lambda}_i| \leq \beta \|M\|,$$
using at most 
$$O\big(n^5 \log (n/\beta\phi)^2 +n^2\log(n/\beta\phi)^2 \log (n \log(n/\beta\phi))    \big)$$
arithmetic operations on a floating point machine with $O(n^2 \log(n/\beta \phi)^2)$ bits of precision. 
\end{corollary}

\subsection{Overview of the Algorithm and Intermediate Results}
\label{sec:introoverview}

The main subroutine of $\eig$, which we call $\oneig$, is a form of \emph{shifted} inverse iteration that on a diagonalizable input $M\in \bC^{n\times n}$ and  an input accuracy parameter $\beta \geq 0$,   produces a $\beta$-forward approximation $\ax{\lambda}\in \bC$ of an eigenvalue of $M$. The precision required to ensure stability of this subroutine and its running time are a function of $n$ and the eigenvector condition number of $M$, i.e. of
$$\kappa_V(M) : = \inf_{V :  M = V D V^{-1}} \|V\| \|V^{-1}\|.  $$ 
The shifting strategy in $\oneig$ crucially relies on a subroutine $\dispec$,  which allows us to estimate the distance of any given point $s\in \bC$ to the spectrum of $M$ (henceforth denoted by $\Spec{M}$) up to relative distance 0.1.  The subroutine $\dispec$ is in itself a form of \emph{unshifted} inverse iteration on $M-s$ and its required precision and running time are also a function of $n$ and $\kappa_V(M)$. 

Once a $\beta$-forward approximation $\ax{\lambda}\in \bC$ of $M$ is obtained, the algorithm calls a subroutine $\decouple$, which essentially uses inverse iteration on $M-\ax{\lambda}$ to find a vector $v\in \bC^n$ which is close to the right eigenvector of $M$ associated to the eigenvalue which is closest to $\ax{\lambda}$. Then, the subroutine $\deflate$ is called to  reduce the problem $M$ to a smaller instance. 

All of the subroutines used in the algorithm require some control on $\kappa_V(M)$, and some additionally require a lower bound on the minimum eigenvalue gap of $M$, i.e.
$$\gap(M):= \min_{i\neq j} |\lambda_i(M)-\lambda_j(M)|.$$
In order for $\eig$ to work on any matrix, we pre-process the input matrix by adding a small random perturbation.\footnote{If $M\in \bC^{n\times n}$ is the input matrix, we run the algorithm on $M+ \gamma G_n$, where $G_n$ is a normalized complex Ginibre matrix and $\gamma>0$ is a function of the desired accuracy and failure probability.   } This was done in \cite{banks2020pseudospectral} to provide general guarantees for the spectral bisection algorithm, and by now  the random matrix literature possesses several results giving high probability quantitative upper bounds on $\kappa_V$  and lower bounds on $\gap$ for the pre-processed matrix \cite{armentano2018stable, banks2021gaussian, banks2020pseudospectral, jain2021real, banks2020overlaps, ge2017eigenvalue, luh2021eigenvectors}. We refer the reader to Section \ref{sec:shatandain} for a detailed discussion. 

Below we elaborate on the main subroutines of $\eig$ and discuss the  technical results proven in this paper. 
\paragraph{Computing the Distance to the Spectrum ($\dispec$).} Let $M\in \bC^{n\times n}$ be a diagonalizable matrix with spectral decomposition
$$M = \sum_{i=1}^n \lambda_i v_i w_i^*,$$
 and fix $s\in \bC\setminus \Spec{M}$.  The main idea behind $\dispec$ is simple: if $u\in \bC^n$ is a vector sampled uniformly at random from the complex unit sphere $\bS^{n-1}$ then  $\|u^* (s-M)^{-m}\|^{-\frac{1}{m}}$ converges (with probability one) as $m$ goes to infinity, to the distance from $s$ to the spectrum of $M$, which we will denote by $\ds{s}{M}$. Indeed: 
 \begin{align}
 \lim_{m\to \infty} \|u^* (s-M)^{-m}\|^{-\frac{1}{m}} & = \lim_{m\to \infty}\Big\|\sum_{i=1}^n (s-\lambda_i)^{-m} u^*v_i w_i^*\Big\|^{-\frac{1}{m}}  \nonumber 
\\ & = \lim_{m\to \infty} \ds{s}{M} \Big\|\sum_{i=1}^n \left(\frac{\ds{s}{M}}{s-\lambda_i}\right)^{m} u^*v_i w_i^*\Big\|^{-\frac{1}{m}} \nonumber    \\ \label{eq:convergencedispec} & = \ds{s}{M}     
 \end{align}
where the last equality holds almost surely. In Section \ref{sec:oneig} we will prove a quantitative version of this fact, and show that when $m = \Omega\big(\log (n \kappa_V(M))\big) $ one obtains an approximation of $\ds{s}{M}$ up to a relative error of 0.1. We will then conclude that $\dispec$ can be implemented with a running time of at most
$$O(\log(n \kappa_V(M))n^2+ \log(n \kappa_V(M))\log \log(n \kappa_V(M)) )$$
arithmetic operations and prove its backward error guarantees, which depend on $\ds{s}{M}$ and $\kappa_V(M)$.  

\paragraph{Finding One Eigenvalue ($\oneig$).} With $\dispec$ in hand, $\oneig$  generates a sequence of complex numbers $s_0, s_1, \dots $ that converges linearly to an eigenvalue of $M$. This sequence is recursively generated as follows: at time $t$, the algorithm uses $\dispec$ to compute an estimate $\dd_t \approx \dist(s_t, \Spec{M})$ with relative error of at most $0.1$. This guarantees that there is at least one eigenvalue of $M$ inside the annulus
\begin{equation}
\label{eq:annulusi}
\calA_{s_t, \dd_t} := \{ z \in \bC : 0.9  \dd_t\leq  |z-s_t|  \leq 1.12  \dd_t \},
\end{equation}
and hence if $\calN_{s_t, \dd_t}$ is a fine enough net of $\calA_{s_t, \dd_t}$ (we will show that nets of six points suffice), we will be able to guarantee that
$$\min_{s\in \calN_{s_t, \dd_t}} \dist(s, \Spec{\hess}) \leq 0.6 \, \dist(s_t, \Spec{\hess}).$$
Given the above guarantee, $\oneig$ then uses $\dispec$ again, now to estimate the distances of the points $s\in \calN_{s_t, \tau_t}$ to the spectrum of $M$, and chooses a point $s\in\calN$ for which
$$\dispec(s, \Spec{M}) \leq \gamma \tau_t$$
for some suitably chosen parameter $\gamma\in (0, 1)$ (we will show that when $\gamma=0.66$ the above inequality is guaranteed for some point in the net). For such an $s$, $\oneig$ sets $s_{t+1}:= s$ and $\tau_{t+1}:= \dispec(s, \Spec{M})$, after which the iteration is repeated (see Figure \ref{fig:net} for an example). 

 \begin{figure}[h]
  \centering
\includegraphics[scale=.22]{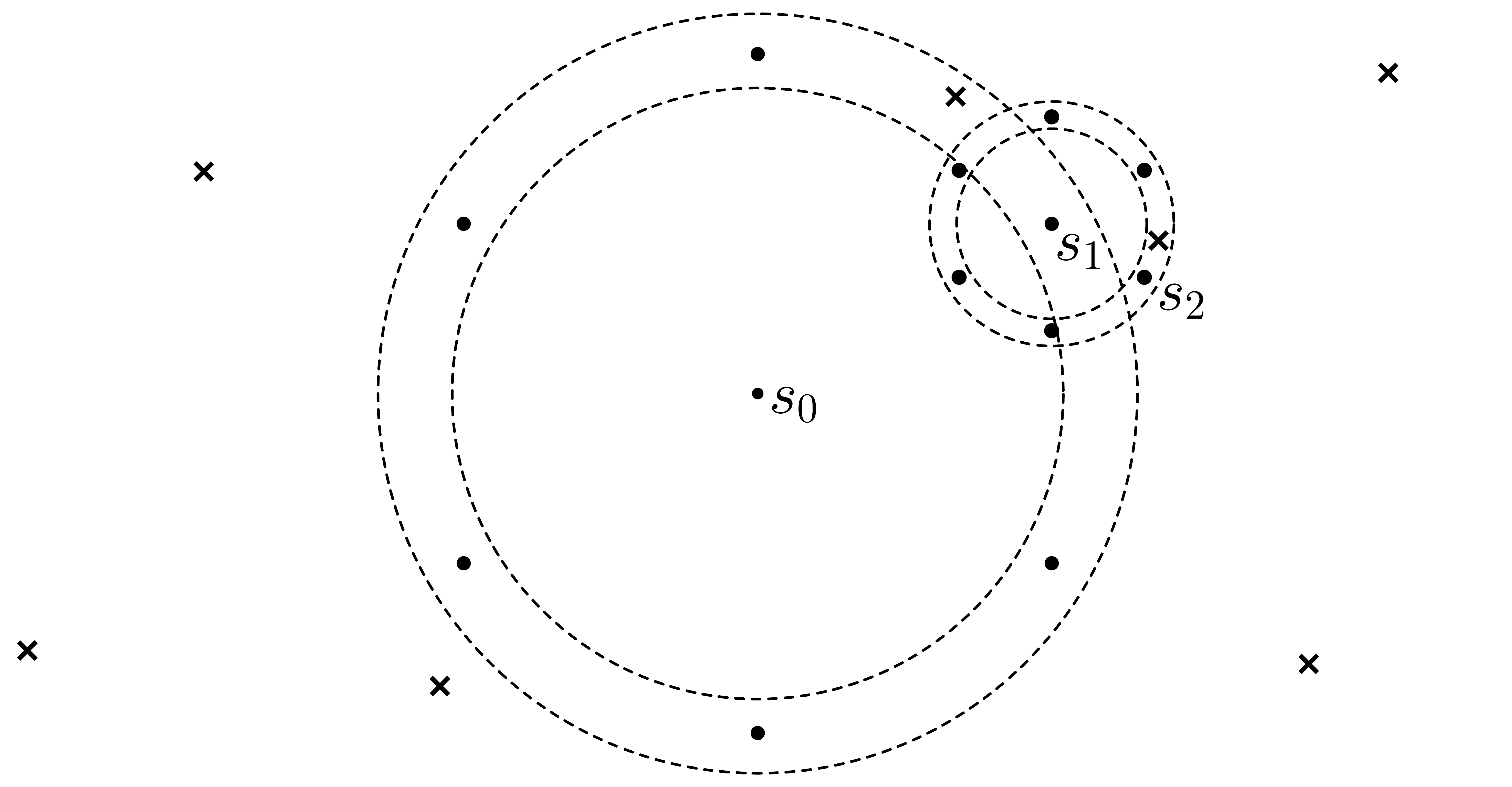}
\caption{The locations of the eigenvalues of $M$ are represented by an $\times$. The figure illustrates the first steps of the iteration which produce $s_0, s_1$ and $s_2$. The annuli $\calA_{s_0, \tau_0}$ and $\calA_{s_1, \tau_1}$ are signal with dotted lines, and the corresponding nets of six points on them are marked.   }
\label{fig:net}
\end{figure}

Clearly, the $s_t$ will converge linearly to an eigenvalue of $M$  and hence finding a point that is at distance at most $\beta$ from the spectrum of $M$ will take $O(\log(1/\beta))$ calls to $\dispec$.  This will be discussed in detail in Section \ref{sec:oneig}. 

\begin{remark}
Intuitively, $\oneig$ is a shifting strategy for inverse iteration where each shift is an \emph{exceptional shift} (cf. \cite{eberlein1975global, wang2002convergence, banks2021global}) chosen from a net of six points. 
\end{remark}

\begin{remark}
Note that even if the subroutine $\oneig$ provides a $\beta$-\emph{forward} approximation of an eigenvalue of the matrix, the ultimate algorithm $\eig$ will only be able to provide an $O(\beta)$-\emph{backward} set of approximate eigenvalues. This is because in order to obtain the full eigendecomposition one needs to \emph{deflate} the problem once a converged eigenvalue is obtained (see the next paragraph for more details on this process), and after deflation we are only able to control the backward error of the eigenvalues that are  subsequently obtained. 
\end{remark}

\paragraph{Implementation of the Subroutines ($ \comptau{m}, \decouple, \deflate$).} There are many ways to implement the subroutines $\dispec$ and $\oneig$ described above. In this paper, for several reasons, we have decided to operate with matrices in their Hessenberg form (similar to what the shifted QR algorithm does). One of the advantages of doing this is that, in the Hessenberg setting, instead of computing the quantity $\|u^* (s-M)^{-m}\|^{-\frac{1}{m}}$ mentioned in the analysis of $\dispec$ one  need  to  compute
$$\tau_{(z-s)^m}(H):= \|e_n^*(s-H)^{-m} \|^{-\frac{1}{m}}$$
where $H$ is a Hessenberg matrix that is unitarily equivalent to $M$ (or \emph{almost} unitarily equivalent when finite arithmetic is taken into account). Computing the latter quantity, as shown in e.g. \cite{banks2022II}, can be done directly from running the implicit QR algorithm $\iqr$ on $H$ (see Section \ref{sec:imported} for a definition of $\iqr$ and the subroutine $\comptau{m}$ defined by it). So in essence, when working with Hessenberg matrices the subroutine $\dispec$ can be easily implemented by calling $\iqr$ with a suitable degree. 

The second advantage of working with a  Hessenberg matrix $H$ is that once a forward approximate eigenvalue $\ax{\lambda}$ of $H$ is found (which is the purpose of $\oneig$), reducing the problem $H$ to a smaller instance becomes easier. Indeed, in Section \ref{sec:decoupling} we will show that if $H_\ell:= \iqr(H, (z-\ax{\lambda})^\ell)$, then one is guaranteed to have $|(H_\ell)_{n, n-1}| = O(\beta)$ for some $\ell = O(n \log \kappa_V(M))$. This will allow us to decouple and then deflate the problem.  

\begin{remark}[Comparison to Shifted QR]
One reason why our algorithm is not an actual shifted QR algorithm is that we have chosen to \emph{maintain} the same Hessenberg matrix $H$ throughout the computation of the shifts $s_1, s_2, \dots$ done by $\oneig$, as opposed to  \emph{updating} the Hessenberg matrix in each iteration to produce a sequence of Hessenberg matrices $H_0=H, H_1, \dots$ hand in hand with the computation of each $s_t$ (as a standard shifted QR algorithm would do). During this process we are using the Hessenberg structure merely as a device for a fast implementation of inverse iteration, and not any of its more subtle properties as in \cite{banks2021global}. Between calls to $\oneig$ the Hessenberg structure is further used to deflate the matrix in a convenient manner.

More substantially, $\oneig$  requires as input a Hessenberg matrix whose right eigenvectors all have {\em reasonably large} (say $1/\poly(n)$) inner products with the vector $e_n$; this is roughly because our analysis is based on the power method and not the more sophisticated potential-based arguments of \cite{banks2021global} which require no assumptions whatsoever. We guarantee the inner product condition by computing a Hessenberg form with respect to a {\em random} vector. Unfortunately  this must be redone after each deflation, which inflicts a cost in the running time of $O(n^4)$, as opposed to the $O(n^3)$ achieved by algorithms that do not need to repeatedly recompute the Hessenberg form.
\end{remark}

\paragraph{Randomness in the Algorithm ($\rhform, \unif(D(0, \eta_2)), G_n$).} Our algorithm uses randomness in three different ways. The first one is related to the inverse iteration described above when discussing $\dispec$. In the Hessenberg setting, the equivalent of running inverse iteration on a randomly chosen vector is to compute a \emph{random} Hessenberg matrix $H$ that is unitarily equivalent to the initial matrix $M$, where the randomness is uniform (in some suitable sense) among the set of Hessenberg matrices that are uniformly equivalent to $M$. The source of randomness in this case is also  a unit vector distributed uniformly on the complex unit sphere $\bS_{\bC}^{n-1}$. We refer the reader to Section \ref{sec:randsampling} for the details  on the sampling assumptions made in this paper, and to Section  \ref{sec:rhess} for an analysis of the subroutine $\rhform$ which on an input matrix $M$ returns a Hessenberg matrix $H$ chosen at random from  the unitary equivalence class (up to machine error) of $M$. 

The second use of randomness is related to the forward stability of $\iqr(H, (z-s)^m)$, which as discussed in \cite{banks2022II}, is a function of $\ds{s}{H}$. As in \cite{banks2022II}, before every call to $\iqr$ we will  add a small random perturbation to the desired shift $s$, i.e. we define $\chs := s+w$ with $w$ chosen uniformly at random from the disk centered at zero of radius $\pretolr$ --- henceforth denoted by $w\sim \unif(D(0, \pretolr))$ --- and run $\iqr(H, (z-\chs)^m)$ instead of $\iqr(H, (z-s)^m)$. The point of doing this is to ensure that with high probability $\ds{\chs}{H}\geq \tol$, for some appropriately chosen (as a function of the desired probability) tolerance parameter $\tol$ that will ultimately determine the precision required for $\iqr$ to be numerically forward stable, a necessary condition for our running time guarantees on $\eig$ to hold.   

Finally, the third way in which we use randomness is to randomly perturb the matrix that is given as input to $\eig$, with the purpose of having  high probability upper and lower bounds on $\kappa_V$ and $\gap$ (cf. \cite[Remark 1.4]{banks2021global} and \cite{banks2020pseudospectral}) when running the subroutines of $\eig$. For this we assume access to a Gaussian sampler that allows us to generate (once) an $n\times n$ complex Ginibre matrix $G_n$. 
\bigskip

To conclude this section we make some comments about our analysis and presentation. 

\paragraph{Pseudospectrum vs $\gap$ and $\kappa_V$.} Although all of the requirements, actions, and guarantees of the subroutines used by the main algorithm can be phrased in terms of the minimum eigenvalue gap and eigenvector condition number of the matrices in question, in some cases we have decided to instead work with the notion of pseudospectrum. This treatment simplifies the analysis of the effects of roundoff error, since the perturbation theory for the pseudospectrum of a matrix is significantly simpler than that for the eigenvalue gap and eigenvector condition number. In Section \ref{sec:pseudospectrum} we include all the necessary preliminaries regarding the notion of pseudospectrum, and explain in what sense the eigenvector condition number and minimum eigenvalue gap of a matrix can be encoded (conversely recovered) in  the pseudospectrum.  

\paragraph{Use of Global Data.} As in \cite{banks2022II} we will use the notion of \emph{global data} when presenting the pseudocode of the algorithms. Here, the global data will be composed of four quantities that all of the subroutines can access if needed. More specifically, the global data will be given by $n$ the dimension of the original input matrix, $\scale$ an approximation of the norm of the  matrix, and  two parameters $\epsilon$ and $\shat$ which will be used to control the pseudospectrum.  

\subsection{Related Work and Discussion}
\label{sec:relatedwork}

Inverse iteration has been used since the 1940's \cite{wielandt1944iterationsverfahren} as a method for computing an eigenvector when an approximation of the corresponding eigenvalue is known; a detailed survey of its history and properties may be found in \cite{varah1968calculation,peters1971calculation,peters1979inverse,ipsen1997computing}. In contrast, this paper uses inverse iteration along with a simple shifting strategy to find the eigenvalues from scratch. 

As discussed in the references above, two situations in which the behavior of inverse iteration in finite arithmetic is known to be tricky to analyze are: (1) matrices with tiny eigenvalue gaps (2) nonnormal matrices which exhibit transient behavior. We deal with these issues by assuming {\em a priori} bounds on the eigenvalue gaps and nonnormality of our input matrix (see Definition \ref{def:shat}) and always dealing with high enough powers of the inverse to dampen transient effects. Assuming such bounds is not restrictive because they may be guaranteed with high probability by adding a small random perturbation, as discussed above.

The algorithm in this paper is, at the time of writing, one of four known provable algorithms for computing backward approximations of the eigenvalues of an arbitrary complex matrix in floating point arithmetic, along with \cite{armentano2018stable,banks2020pseudospectral,banks2022II}. The strengths of the algorithm are its simplicity and use of $O(\log^2(n/\delta))$ bits of precision, which is better than \cite{banks2020pseudospectral} but worse than \cite{armentano2018stable} (however $\cite{armentano2018stable}$ has the drawback of running in $O(n^{10}/\delta)$ arithmetic operations). The main weakness of this algorithm compared to  \cite{banks2020pseudospectral, banks2022II} is its use of $O(n^4)$ arithmetic operations for repeatedly computing the Hessenberg form.
  We do not know any example where this recomputation after deflation is actually needed, but are not able to prove that it is not (with high probability). Doing so would entirely remove the $O(n^4)$ factor from the running time in Theorem \ref{thm:main} and is worthy of further investigation.

%% file: preliminaries.tex
\section{Preliminaries}

\label{sec:preliminaries}

As in the previous two papers in this sequence, all vector/matrix norms are $\ell_2$/operator norms unless stated otherwise,  and we use the notation 
$$\dist(\calS_1, \calS_2) := \inf_{s_1 \in \calS_1, s_2\in \calS_2} |s_1-s_2|. $$
for any sets  $\calS_1, \calS_2\subset \bC$, and when $s\in \bC$ we use $\dist(s, \calS)$ as a shorthand notation for $\dist(\{s\}, \calS)$. 

\subsection{Finite Precision Arithmetic}
We use the standard floating point axioms from \cite[Chapter 2]{higham2002accuracy} (ignoring overflow and underflow as is customary), and use $\mach$ to denote the unit roundoff. Specifically, we will assume that we can add, subtract, multiply, and divide floating point numbers, and take square roots of positive floating point numbers, with relative error $\mach$. We will use $\fl(\ast)$ to denote that the expression $\ast$ is computed in finite arithmetic.

As in \cite{banks2022II} we will have to compute $m$-th roots of positive numbers, for which we assume access to an algorithm satisfying the guarantees of the following lemma.  

\begin{lemma}[Lemma 2.1 in \cite{banks2022II}]
\label{lem:mthroot}
There exist small universal constants $\Croot, \croot \geq 1$, such that whenever $m \croot  \mach \leq \epsilon \leq 1/2 $ and for any $a\in \bR^+$, there exists an algorithm that computes $a^{\frac{1}{m}}$ with relative error $\epsilon$ in at most $$T_{\rt}(m,\epsilon):= \Croot m \log(m\log(1/\epsilon))$$ arithmetic operations.
\end{lemma}

\subsection{Random Sampling Assumptions.} 
\label{sec:randsampling}

In Section \ref{sec:introoverview} we enlisted the three different ways in which randomness is used in $\eig$. Here we specify the assumptions we make about the algorithms used to generate the desired random objects. 

\begin{definition}[Efficient $\unif(\bS^{n-1}_{\bC})$ Sampler]
\label{def:usampler}
An efficient random vector algorithm takes as input a positive integer $n$ and  generates a random unit vector $u\in \bC^n$ distributed uniformly in the complex unit $n$-sphere $\bS^{n-1}$ and runs in $\cu n$ arithmetic operations, for some universal constant $\cu$. 
\end{definition}

\begin{definition}[Efficient $\unif(D(0, R))$ Sampler]
An efficient random perturbation algorithm takes as input an $R>0$, and  generates a random $w\in \bC$ distributed uniformly in the disk $D(0, R)$, and runs in $\cd$ arithmetic operations, for some universal constant $\cd$.   
\end{definition}

\begin{definition}[Efficient Ginibre Sampler]
An efficient Ginibre sampler takes as input a positive integer $n$ and generates a random matrix $G_n\in \bC^{n\times n}$, where the entries of $G_n$ independent centered complex Gaussians of variance $1/n$, and runs in $\Cg n^2$ arithmetic operations.  
\end{definition}

Note that the roundoff error in the algorithm coming from using finite precision when sampling any of these  random objects only affects (in a negligible way) the failure probabilities reported in the analysis of the algorithm, and not the quantities handled by the algorithm itself. So, for simplicity we will assume that the samples can be drawn from their exact distribution.

\subsection{Definitions and Lemmas from \cite{banks2021global} and \cite{banks2022II}.}
\label{sec:imported}

\paragraph{Approximate Functional Calculus.} As in the first two parts of this series, we will exploit the notion of \emph{approximate functional calculus}. For a diagonalizable Hessenberg matrix $H\in \bC^{n\times n}$, with  diagonalization $H=VDV^{-1}$ for $V$ chosen\footnote{If there are multiple such $V$, choose one arbitrarily.} to satisfy $\|V\|= \|V^{-1}\| = \sqrt{\kappa_V(H)}$, define $Z_H$ to be the random  variable supported on $\Spec(H)$ with distribution
\begin{equation}\label{eqn:specmeasure}\P[Z_H = \lambda_i ] = \frac{|e_n^* V e_i|^2}{\|e_n^* V\|^2},\end{equation}
where $\lambda_i= D_{ii}$. As in the prequels, we  will often use the following inequalities (see \cite[Lemma 2.4]{banks2021global} for a proof).

\begin{lemma}[Approximate Functional Calculus]
    \label{lem:spectral-measure-apx}
    For any upper Hessenberg $H$ and complex function $f$ whose domain includes the eigenvalues of $H$,
    $$
        \frac{\|e_n^\ast f(H)\|}{\kappa_V(H)} \le \E\left[|f(Z_H)|^2\right]^{\frac{1}{2}} \le \kappa_V(H)\|e_n^\ast f(H)\|.
    $$
\end{lemma}

\paragraph{Implicit QR Algorithm.} For an invertible matrix $M$ we will use $[Q, R] = \qr(M)$ to denote that $M=QR$ is the unique QR decomposition of $M$ where the upper triangular part $R$ has  positive diagonal entries. 

We will assume access to a degree 1 implicit QR algorithm $\iqr(H, s)$, which is  $\muqr(n)$-stable in the sense of \cite[Definition 3.4 ]{banks2022II} and we will implement higher degree shifts by composing this $\iqr$ algorithm, that is, for any polynomial $p(z) = (z-s_1)\cdots (z-s_m)$ we define 
$$\iqr(H,p(z)):= \iqr(\iqr(\cdots \iqr(\iqr(H,s_1),s_2), \cdots), s_m),$$
and recall the following backward-stability guarantees given in \cite[Lemma 3.6]{banks2022II}. 

\begin{lemma}[Backward Error Guarantees for $\iqr$]
    \label{lem:iqr-multi-backward-guarantees}
    Fix $C > 0$ and let $p(z) = \prod_{\ell \in [m]}(z - s_\ell)$, where $\calS = \{s_1,...,s_m\} \subset \mathbb{D}(0,C\|H\|)$. If $\ax{\next{H}} = \iqr(H,p(z))$, and
    $$
        \muqr(n)\mach \le 1/4,
    $$
    there exists a unitary $\ax{Q}$ satisfying
    \begin{align}
        \left\|\ax{\next{H}} - \ax{Q}^\ast H \ax{Q} \right\| \le 1.4 m(1 + C)\|H\|\muqr(n)\mach. 
    \end{align}
\end{lemma}

Using Givens rotations, $\iqr(H, p(z))$ can be executed in 
$$T_{\iqr}(n, m) := 7mn^2$$ 
arithmetic operations and it is $\muqr(n)$-stable for $\muqr(n) = 32n^{3/2}$ (see \cite[Appendix A]{banks2022II} for details). Forward error guarantees for $\muqr(n)$-stable implicit QR algorithms on an input $H \in \bC^{n\times n}$ can also be given, this time in terms of the distance of the shifts to the spectrum of $H$.   More precisely, the following part of Lemma 3.9 in \cite{banks2022II} will be used repeatedly below.

\begin{lemma}[Forward Error Guarantees for $\iqr$]
    \label{lem:multiiqrstability}
    Let $H\in \bC^{n\times n}$ be a Hessenberg matrix and fix $C > 0$. Assume that $p(z) = \prod_{\ell \in [m]}(z - s_\ell)$, where  $\mathcal{S} = \{s_1, \dots, s_m \} \subset D(0,C\|H\|)$.  Furthermore, let $[Q, R] = \qr(p(H))$, $\next{H} = Q^\ast H Q$, and assume that
    \begin{align}
    \label{assum:machvsp}
        \mach \le \mach_{\iqr}(n,m,\|H\|,\kappa_V(H),\dist(\calS,\Spec{H})) 
        &:= \frac{1}{8\kappa_V(H)\muqr(n)}\left(\frac{\dist(\calS,\Spec{H})}{\|H\|}\right)^m   \\
        &= 2^{-O\left(\log n\kappa_V(H) + m\log\frac{\|H\|}{\dist(\calS,\Spec{H})}\right)}. \nonumber
        \end{align}
    Then, we have the forward error guarantees:
$$
       \left\|\ax{\next{H}} - \next{H}\right\|_F  \le 32\kappa_V(H) \|H\|\left(\frac{(2 + 2C)\|H\|}{\dist(\calS,\Spec{H})}\right)^m n^{1/2}\muqr(n)\mach. 
   $$
\end{lemma}

\paragraph{Computing $\tau^m$.} For a Hessenberg matrix $H\in \bC^{n\times n}$ and $s\in \bC$, our algorithm needs to estimate quantities of the form $\|e_n^* (s-H)^{-m}\|^{-1}$. For this task we will use the subroutine $\comptau{m}$ which was analyzed in \cite{banks2022II}. 
\bigskip

\begin{boxedminipage}{\textwidth}
$$\comptau{m}$$
    \textbf{Input:} Hessenberg $H\in \bC^{n\times n}$, polynomial $p(z)=(z-s_1)\cdots (z-s_m)$ \\
    \textbf{Output:} $\nn \geq 0$ \\
    \textbf{Ensures:} $|\nn - \tau_p (H)^m| \le 0.001 \tau_p(H)^m$ 
\begin{enumerate}
    \item  $[\ax{\hat{H}},  \ax{R}_1, \dots, \ax{R}_m] \gets \iqr (H, p(z))$
    \item $\nn \gets \fl\left( (\ax{R}_1)_{nn}\cdots (\ax{R}_m)_{nn} \right)$
\end{enumerate} 
\end{boxedminipage}
\bigskip

\begin{lemma}[Lemma 3.9 in \cite{banks2022II}] 
\label{lem:guaranteetaum}
    If $\calS = \{s_1,...,s_m\} \subset \mathbb{D}(0,C\|H\|)$ and
    \begin{align}
        \label{assum:comptau}
        \mach 
        &\le \mach_{\comptau{}}(n,m,C,\|H\|,\kappa_V(H),\dist(\calS,\Spec{H})) \\
        &:= \frac{1}{6 \cdot 10^3 \kappa_V(H) \muqr(n)}\left(\frac{\dist(\calS,\Spec{H})}{(2 + 2C)\|H\|}\right)^{2m} \nonumber  \\
        &= 2^{-O\left(\log n\kappa_V(H) + m\log \frac{\|H\|}{\dist(\calS,\Spec{H})}\right)}, \nonumber
    \end{align}
    then $\comptau{m}$ satisfies its guarantees, and runs in $$T_{\comptau{}}(n,m):= T_{\iqr}(n, m) + m = O(m n^2)$$ arithmetic operations.
\end{lemma}

\paragraph{Shift Regularization.} In this paper we will only call $\comptau{m}$ on polynomials of the form $p(z)=(z-s)^m$ for some $s\in \bC$. So, proceeding as in \cite{banks2022II}, to have a control on the relative accuracy of $\comptau{m}$, we will randomly perturb $s$ to ensure that it is far enough from the spectrum of the input matrix. To be precise, we will use the following particular case of \cite[Lemma 3.10]{banks2022II}. 

\begin{lemma}[Regularization of Shifts]
    \label{lem:fixguarantee1} 
    Let $s \in \bC$  and  $\pretolr\geq \tolr >0$, and  assume that $ \tolr +\pretolr \leq \frac{\gap(H)}{2}.$
    Let $\w \sim \text{Unif}(D(0, \pretol))$  and $\chs := s+w$. Then with probability at least $1-\left(\tol/\pretol\right)^2$, we have $\dist(\chs,\Spec H) \ge \tolr$.
\end{lemma}

\subsection{Pseudospectrum}
\label{sec:pseudospectrum}

Given $M\in \bC^{n\times n}$ and $\epsilon>0$ the $\epsilon$-pseudospectrum of $M$ is defined as 
\begin{equation}
   \label{eqn:pseudodef2}
   \Lambda_\epsilon(M): = \left\{\lambda \in \bC : \big\|(\lambda - M)^{-1}\big\| \geq 1/\epsilon \right\}.
\end{equation}
In particular $\Spec(M)\subset \Lambda_\epsilon(M)$ for every $\epsilon>0$, and one can show (see \cite{trefethen2020spectra})  that 
\begin{equation*}
    \label{eq:characterizationofLambda}
    \Lambda_\epsilon(M) = \{\lambda \in \bC : \lambda\in \mathrm{Spec}(M+E) \text{ for some } \|E\|\leq \epsilon\}, 
\end{equation*}
and as direct consequence the following two standard properties follow. 

\begin{lemma}
\label{lem:decrementeps} 
    For any $M, E, U\in \bC^{n\times n}$ with $\|E\|\leq \epsilon$ and $U$ unitary, the following are true 
    \begin{enumerate}[label=\roman*)]
        \item $\Lambda_{\epsilon}(UM U^*) = \Lambda_\epsilon(M)$. 
        \item $\Lambda_\epsilon(M+E) \subset \Lambda_{\epsilon-\|E\|}(M)$. 
    \end{enumerate}
\end{lemma}

We refer the reader to the excellent book \cite{trefethen2020spectra} for a comprehensive treatment on the notion of pseudospectrum. For this paper we will only need the following basic lemmas that relate the pseudospectrum to the notions of eigenvalue gap and eigenvector condition number. First, we recall that the pseudospectrum can be controlled in terms of the eigenvector condition number. 

\begin{lemma}[\cite{trefethen2020spectra}]
\label{lem:pseudospectralbauerfike}
   For every $M \in \C^{n\times n}$,
    \begin{equation} \label{eqn:lambdakappa}
        \bigcup_i D(\lambda_i,\eps)\subset \Lambda_\eps(M)\subset \bigcup_i D(\lambda_i, \eps \kappa_V(M)).
    \end{equation}
\end{lemma}

When analyzing the algorithm in finite arithmetic it will be necessary to have some control on the eigenvector condition number and minimum eigenvalue gap of the matrices produced by the algorithm. For this, we will use the notion of $\shat$-shattered pseudospectrum, which is very similar to the notion of shattered pseudospectra introduced in \cite{banks2020pseudospectral}, but without referencing a grid. 

\begin{definition}[$\shat$-shattered pseudospectrum]\label{def:shat}
Let $\epsilon, \shat>0$ and $M \in \bC^{n\times n}$. We say that $\Lambda_\epsilon(M)$ is $\shat$-shattered if there exist $n$ disjoint disks $D_1, \dots, D_n$ of radius $\shat$ such that
\begin{enumerate}[label=\roman*)]
    \item (Containment) $\Lambda_\epsilon(M)\subset \bigcup_{i=1}^n D_i.$
    \item (Separation) Any two disks are at distance at least $\shat$, that is, $\dist(D_i, D_j)\geq \shat$ for all $i\neq j$.    
\end{enumerate}
\end{definition}

In what can be thought as a converse of Lemma \ref{lem:pseudospectralbauerfike}, the shattering parameter can be used to control the eigenvector condition number of a matrix and its minimum eigenvalue gap. 

\begin{lemma}[$\kappa_V$ from $\shat$ and $\epsilon$]
\label{lem:kappavfromshattering}
Let  $\epsilon, \shat>0$ and $M\in \bC^{n\times n}$. If $\Lambda_\epsilon(M)$ is $\shat$-shattered, then 
\begin{enumerate}[label=\roman*)]
    \item  $\kappa_V(M)\leq \frac{n \shat}{\epsilon}$. 
    \item  $\gap(M)\geq \shat$. 
\end{enumerate}

\end{lemma}

\begin{proof}
First note that ii) follows from the fact taht $\Spec(M)\subset \Lambda_\epsilon(M)$ and the definition of $\shat$-shattering. To show i) let $\lambda_1, \dots, \lambda_n$ be the eigenvalues of $M$, and for every $i$ let $\kappa(\lambda_i)$ the denote the eigenvalue condition number of $\lambda_i$ (see \cite[Section 2.2]{banks2020pseudospectral} for a definition). A trivial modification of the proof of Lemma 3.11 in \cite{banks2020pseudospectral} yields that $\kappa(\lambda_i) \leq \frac{\shat}{\epsilon}$. Then, by Lemma 3.1 in \cite{banks2021gaussian} we have 
$$\kappa_V(M) \leq \sqrt{n\sum_{i=1}^n \kappa(\lambda_i)^2} \leq \frac{n\shat}{\epsilon}.$$
\end{proof}

%% file: oneig.tex
\section{The Shifting Strategy}
\label{sec:shiftingstrategy}

\subsection{Analysis of $\dispec$}
\label{sec:dispec}

We define the subroutine $\dispec(H, s, m)$ as follows and prove its guarantees below. 
\bigskip

\begin{boxedminipage}{\textwidth}
$$\dispec$$
    \textbf{Input:} Hessenberg $\hess\in \bC^{n\times n}$, $s\in\bC$,  $m\in \bN$    \\
    \textbf{Output:} $\dd \geq 0$ \\
    \textbf{Ensures:} $ \frac{0.998}{\kappa_V(\hess)^{\frac{1}{m}}} \ds{s}{\hess} \leq \dd \leq  \frac{1.003 \kappa_V(\hess)^{\frac{1}{m}} }{\P\big[|Z_\hess-s|=\ds{s}{\hess}\big]^{\frac{1}{2m}}} \ds{s}{\hess}$
\begin{enumerate}
    \item $\nn \gets \comptau{m} (\hess, (z-s)^m)$
    \item $\dd \gets \fl\left( (\nn)^{\frac{1}{m}}\right)$
\end{enumerate} 
\end{boxedminipage}

\bigskip

\begin{proposition}[Guarantees for $\dispec$]
\label{prop:guaranteefordispec}
Let $\rad >0$ and assume that $s \in D(0, C\|\hess\|)$. Then, the algorithm $\dispec$ runs in $$T_{\dispec}(n, m):= T_{\comptau{}}(n, m)+ T_{\rt}(m, 10^{-3}) = O(mn^2+m\log m) $$ 
arithmetic operations and satisfies its guarantees provided that
    \begin{align}
    \label{assum:dispec}
        \mach & \leq \mach_{\dispec}(n,m,C,\|H\|,\kappa_V(H),\dist(s,\Spec{H}))
        \\ & :=  \frac{1}{\croot} \mach_{\comptau{}}(n,m,C,\|H\|,\kappa_V(H),\dist(s,\Spec{H})). \nonumber
    \end{align}
\end{proposition}

\begin{proof} First note that  
\begin{align}
\tau_{(z-s)^m}(\hess) & = \|e_n^* (\hess-s)^{-m}\|^{-\frac{1}{m}}  \nonumber
\\ & \le \frac{\kappa_V(H)^{\frac{1}{m}}}{\E\left[|Z_H-r|^{-2m}\right]^{\frac{1}{2m}}} \nonumber && \text{Lemma \ref{lem:spectral-measure-apx}} \nonumber \\ \label{eq:boundontau}
        &\le \frac{\kappa_V(H)^{\frac{1}{m}}\ds{r}{H}}{\P\Big[|Z_H-r| =   \ds{r}{H}\Big]^{\frac{1}{2m}}}.
\end{align}
Similarly, to lower bound  $\tau_{(z-s)^m}(\hess)$ use Lemma \ref{lem:spectral-measure-apx} again to obtain 
$$ \tau_{(z-s)^m}(\hess)= \|e_n^* (\hess-s)^{-m}\|^{-1/m} \geq  \frac{1}{\kappa_V(\hess)^{\frac{1}{m}}\E[|Z_\hess-s|^{-2m}]^{\frac{1}{2m}}}\geq \frac{\ds{s}{\hess}}{\kappa_V(\hess)^{\frac{1}{m}}}.$$
So, it only remains to control $|\tau-\tau_{(z-s)^m}(H)|$, where $\tau$ is the output of $\dispec$. Since by assumption (\ref{assum:dispec}) holds, we can apply Lemma \ref{lem:guaranteetaum} to get
\begin{equation*}
\label{eq:intermediateestimation}
0.999 \tau_{(z-s)^m}(\hess)^m \leq \nn \leq 1.001\tau_{(z-s)^m}(\hess)^m.
\end{equation*}
Similarly, we can apply Lemma \ref{lem:mthroot} to get that $\fl((\nn)^{\frac{1}{m}})$ can be computed to relative accuracy $\epsilon = 10^{-3}$, using at most $T_{\rt}(m, 10^{-3})$ arithmetic operations. Hence
$$0.999 (\nn)^{\frac{1}{m}}  \leq \fl((\nn)^{\frac{1}{m}}) \leq 1.001 (\nn)^{\frac{1}{m}},$$
which combined with all of the above yields the advertised guarantees. To compute the final running time, add to  $T_{\rt}(m, 10^{-3})$ the $T_{\comptau{}}(n,m)$ arithmetic operations needed to compute $\comptau{m}$. 
\end{proof}

\subsection{Analysis of $\oneig$}
\label{sec:oneig}

For every $s\in \bC$ and $\tau>0$, on the annulus $\calA_{s, \tau} =\{ z \in \bC : 0.9  \tau \leq  |z-s|  \leq 1.12  \tau \}$  we will define  the set $\calN_{s, \tau}$ of six points given  by
$$\calN_{s, \tau} := \left\{ s+ \tau e^{i \pi  \ell/3} : \ell=1, \dots, 6 \right\}.$$
As explained in Section \ref{sec:introoverview},  at time $t$,  $\oneig$  will call $\dispec$ on the the locations given by the points in a net on $\calA_{s_t, \tau_t}$ for some $s_t$ and $\tau_t$. So, to give accuracy guarantees on the output provided by $\dispec$, we will choose the net to be  the randomly perturbed set  $$\chn_{s_t, \tau_t} := \{s_t+w: s_t\in \calN_{s_t, \tau_t} \}, \quad \text{where}\quad w\sim \unif(D(0, \pretolr)),$$   (cf. the discussion on shift regularization in Section \ref{sec:imported}).

We begin by noting that for any $s\in \bC$ and $\tau>0$,  $\chn_{s, \tau}$ is a net on $\calA_{s, \tau}$ in the following sense. 

\begin{observation}
\label{obs:scalereduction}
Using the above notation, if $\pretolr  \leq .03 \tau$ then for any realization of $\chn_{s, \tau}$ we have  $$\sup_{z\in \calA_{s, \tau}}\mathrm{dist}\big(z, \chn_{s, \tau}\big)\leq 0.6 \tau.$$
\end{observation}

\begin{proof}
Basic trigonometry shows that because $z\in \calA_{s, \tau}$ we can guarantee $\mathrm{dist}(z, \calN_{s, \tau})\leq .57 \tau.$ Then, because any realization of $w\sim D(0, \pretolr)$ (which yields a realization of $\chn_{s, \tau}$) satisfies $|w| \leq \pretolr  \leq .03 \tau$, the result follows from the triangle inequality. 
\end{proof}

We can now define the algorithm. 
\bigskip

\begin{boxedminipage}{\textwidth}
$$\oneig$$
    \textbf{Input:} $\hess\in \bC^{n\times n}$ Hessenberg, accuracy $\fward>0$, failure probability tolerance $\wfail$, eigenvalue mass lower bound $p$     \\
    \textbf{Global Data:} Norm bound $\scale$,  pseudospectral parameter $\epsilon$, shattering parameter $\shat$ \\
    \textbf{Output:} $[\ax{\lambda}, \cor]$ with $\ax{\lambda}\in \bC$ and $\cor \in \{\true, \false\}$   \\
    \textbf{Requires:} $\fward \leq 1/2$, $\Lambda_\epsilon(\hess)$ is $\shat$-shattered, $\P[Z_{\hess} = \lambda ]\geq p$ for all $\lambda \in \Spec{\hess}$, $ 10 \fward  \leq \|\hess\| \leq 2 \scale   $ \\ 
    \textbf{Ensures:} With probability at least $1-\wfail$, $\oneig$ terminates successfully, that is $\cor =\true$ and $\ax{\lambda}$ satisfies   $ \tolr \leq \ds{\ax{\lambda}}{\hess} \leq  \fward $, where $\tolr$ is defined in line \ref{line:mandeta}
\begin{enumerate}
\item \label{line:mandeta} $m \gets \left\lceil 12 \left( \log\left(\frac{n\shat}{\epsilon}\right) + \frac{1}{2} \log\left(\frac{1}{p}\right) \right) \right\rceil$,  $\pretolr\gets \frac{\fward}{5}\wedge \frac{\shat}{3}$, $\tolr \gets \pretolr  \left( \frac{\wfail}{12 \log(3\scale/10\fward)} \right)^{1/2}$
    \item \label{line:oneiginitialization} $w\sim \unif(D(0, \pretolr )), \, \chs \gets \hess_{nn}+w, \, \dd \gets \dispec(\chs , \hess, m)$
    \item \label{line:oneigwhileloop} \textbf{While} $\dd > 0.9 \fward $ 
   
    \begin{enumerate}
    \item     $w\sim \unif(D(0, \pretolr  ))$,    $\, \chn \gets \{\chs^{(1)}, \dots, \chs^{(6)}\}=  \calN_{\chs, \tau}+w$ 
    \item $\tau' \gets \min_{j \in [6]} \dispec(\chs^{(j)},H,m)$
    \item \textbf{If} $\tau' \le 0.66\tau$\\
     $\chs \gets \chs^{(j)}$, $\tau \gets \tau'$, $\cor \gets \true$
    \item \textbf{Else} $\cor \gets \false$, terminate $\oneig$ and output $[\chs, \false]$.

    \end{enumerate}     
     \item $\ax{\lambda} \gets \chs$, output $[\ax{\lambda}, \true]$
\end{enumerate} 
\end{boxedminipage}
\bigskip

\begin{remark}[About the $\cor$ Flag]
Although small, there is a positive probability that while running $\oneig$ the subroutine $\dispec$ is called on a complex number $s\in \bC$ for which $\ds{s}{H} < \tolr$. When this  happens there will be no guarantee that the output of $\dispec$ is relatively accurate, and the information provided by it might be misleading, giving rise to an update of $\chs$ for which the distance to $\Spec{H}$ might be even larger than what it was for its previous value.  In view of this, the purpose of the flag $\cor$
is to identify when as a consequence of an inaccurate output of $\dispec$ it is no longer possible to decrease the variable $\dd$ at a geometric rate, in which case the algorithm halts and outputs $\error$\footnote{Of course, one could try to formulate a dichotomy as in \cite{banks2022II} in which one leverages that errors can only be made once the shifts that are being used are very close to $\Spec{H}$, and have a mechanism that outputs a forward approximate eigenvalue even when $\dispec$ provides inaccurate answers. Since this proved to be intricate, for the sake of clarity we have decided to settle for this simpler, but efficient enough, version of the algorithm.}.
\end{remark}

Before proving the main result about $\oneig$, we observe that in line \ref{line:mandeta} of this algorithm, $m$ is set so that $\dispec(s, H, m)$ will yield an accurate approximation of $\ds{s}{\hess}$ all throughout the iteration (provided that $s$ is not too close to $\Spec{H}$). 

\begin{observation}[$m$ is large enough]
\label{obs:gooddistapprox}
Let $\rad>0$, $s\in D(0, \rad \|\hess\|)$ and $m$ be as in line \ref{line:mandeta} of $\oneig$. Assume that the requirements of $\oneig$ are satisfied and that
\begin{equation}
\label{eq:assumdispecforoneig}
\mach \leq \mach_{\dispec}(n,m,C,\|H\|,\kappa_V(H),\dist(s,\Spec{H})).
\end{equation}
 Then 
$$0.9 \ds{s}{\hess} \leq \dispec(\hess, s, m) \leq 1.1 \ds{s}{\hess}.$$
\end{observation}
\begin{proof}
Let $\dd = \dispec(\hess, s, m)$.  Since $\mach\leq \mach_{\dispec}$ we can apply Proposition \ref{prop:guaranteefordispec} to get  
$$ \frac{0.998}{\kappa_V(\hess)^{\frac{1}{m}}} \ds{s}{\hess} \leq \dd \leq  \frac{1.003 \kappa_V(\hess)^{\frac{1}{m}} \ds{s}{\hess}}{\P\big[|Z_\hess-s|=\ds{s}{\hess}\big]^{\frac{1}{2m}}} .$$
Then, it suffices to show that $$0.9 \leq \frac{0.998}{\kappa_V(\hess)^{\frac{1}{m}}} \quad \text{and} \quad  \frac{1.003 \kappa_V(\hess)^{\frac{1}{m}} }{\P\big[|Z_\hess-s|=\ds{s}{\hess}\big]^{\frac{1}{2m}}}\leq 1.1,$$
or equivalently
$$m \geq \frac{\log(\kappa_V(\hess))}{\log\left(0.998/0.9\right)} \quad \text{ and } \quad  m \geq \frac{\log(\kappa_V(\hess))+ \frac{1}{2} \log(1/ \P\big[|Z_\hess-s|=\ds{s}{\hess}\big])}{\log(1.1/1.003)}.$$
Finally, using that 
$$\P\big[|Z_\hess-s|=\ds{s}{\hess}\big] \geq \min_{\lambda \in \Spec{H}} \P[Z_\hess=\lambda]\geq p$$ and $\kappa_V(\hess) \leq \frac{n\shat}{\epsilon}$ (which follows from Lemma \ref{lem:kappavfromshattering}), it is clear that this $m$ satisfies the above inequalities. 
\end{proof}

Now we observe that in line \ref{line:mandeta} of $\oneig$, the parameters $\tolr$ and $\pretolr$ are set to be small enough that we can apply Lemma \ref{lem:fixguarantee1}. 

\begin{observation}
\label{obs:etasaresmall}
Let $\tolr, \pretolr$ be as in line \ref{line:mandeta} and assume that the requirements of $\oneig$ are satisfied. Then 
$$\tolr+\pretolr \leq \frac{\gap(H)}{2} \quad \text{and} \quad \pretolr \leq 0.02 \|H\|. $$
\end{observation}

\begin{proof}
Since $\Lambda_\epsilon(H)$ is $\shat$-shattered we have $\shat \leq \gap(H)$, and by definition of the parameters we have $2\tolr \leq \pretolr \leq \shat/3  $, from where $\tolr+\pretolr \leq \gap(H)/2$. To prove the other assertion, note that the requirements of $\oneig$ imply that $\fward \leq 0.1 \|H\|$, on the other hand by definition $\tolr \leq \fward/5$, so the proof is concluded by combining both bounds. 
\end{proof}

We  now state  the main result of this section. 

 \begin{proposition}[Guarantees for $\oneig$]
 \label{prop:findone}
 Assume that the requirements of $\oneig$ are satisfied, let $m$ and $\tolr$ be as defined in line \ref{line:mandeta} of $\oneig$ and assume  that  
\begin{align}
    \label{assum:oneig}  \mach & \leq \mach_{\oneig}(n,\scale, \epsilon, \shat, p, \beta, \varphi) 
        \\ & := \mach_{\dispec}\big(n,m, 10,2 \scale, n\shat/\epsilon,\tolr\big). \nonumber 
\end{align}
Then, with probability at least $1-\wfail$, $\oneig$ outputs a  $\ax{\lambda}\in \bC$ satisfying 
 \begin{equation}
 \label{eq:lambdaguarantees}
 \tolr \leq \ds{\ax{\lambda}}{\hess} \leq \fward, 
 \end{equation}
 using at most  
 \begin{align*}
 T_{\oneig}(n, \scale, \epsilon, \shat, p, \fward ) & := (6\lceil 2 \log(\scale/5\fward)\rceil+1) T_{\dispec}(n, m)+ \lceil 2 \log(\scale/5\fward)\rceil(\cd + 16)+O(1) 
 \\ & = O\big( \log(\scale/\beta) \log(n\shat/\epsilon p) (n^2+ \log \log(n\shat/\epsilon p)) \big)
 \end{align*}
 arithmetic operations. 
 \end{proposition}

Since the proof of this proposition requires several steps we will present it in a separate subsection. 

\subsubsection{Proof of Proposition \ref{prop:findone}}

It is clear that the exact arithmetic version of $\oneig$ would satisfy the advertised guarantees. The challenge is in arguing that in finite arithmetic, with high probability, each call to $\dispec$ yields an accurate enough answer, and that the aggregate  roundoff errors and failure probabilities is not too large.  Since $\dispec$ is based on the subroutine $\iqr$,  inaccuracies can only arise when the input  $s\in \bC$ is either too close to $\Spec{H}$ or $|s|$ is too large. This is quantified in the following observation, which we will use repeatedly throughout the proof. 

\begin{observation}[Conditions for accuracy]
\label{obs:accuracy}
 For any $s\in D(0, 10\|H\|)$ with $\ds{s}{H}\geq \tol$ the following guarantee holds
$$0.9 \ds{s}{H} \leq \dispec(H, s, m) \leq 1.1 \ds{s}{H}.$$
\end{observation}

\begin{proof}
Since $\Lambda_\epsilon(H)$ is $\shat$-shattered by assumption,  Lemma \ref{lem:kappavfromshattering} shows that $\kappa_V(H) \leq \frac{n\shat}{\epsilon}$, and using the assumption $\|H\|\leq 2\scale$, we get that (\ref{assum:oneig}) implies $$\mach \leq \mach_{\dispec}\Big(n,m, 10,\|H\|,\kappa_V(H),\tolr\Big).$$ So, for any $s\in D(0, 10 \|H\|)$ with $\ds{s}{H}\geq \tolr$, $\mach$ will satisfy  inequality \eqref{eq:assumdispecforoneig}, which by Observation \ref{obs:gooddistapprox} yields the desired inequalities. 
\end{proof}

Let $s_0, s_1, \dots$ be the values acquired by the variable $\chs$ throughout the algorithm,  $\tau_0, \tau_1, \dots $ be the values acquired by $\tau$, and $w_0, w_1, \dots $ be the values acquired by $w$. We will now show that, by the structure of the algorithm, the only real obstruction to obtaining accuracy is the possibility of the $s_i$ being to close to $\Spec{H}$. 

\begin{lemma}[Accuracy of the $\tau_i$]
\label{lem:tauaccuracy}
Let $t\geq 0$ and assume that $\oneig$ does not terminate in the first $t$ while loops\footnote{Here, terminating in the while loop $t=0$ means that that the first while loop was never started. }, and that $\ds{s_i}{H}\geq \tol$ for all $i=0, \dots, t$. Then, for all $i=0, \dots, t$ we have that
\begin{equation}
\label{eq:guaranteefori}
0.9 \ds{s_i}{H} \leq \tau_i \leq 1.1 \ds{s_i}{H},
\end{equation}
  $s_i\in D(0, 10\|H\|)$, and moreover $\chn_{s_i, \tau_i}\subset D(0, 10\|H\|)$.  
\end{lemma}

\begin{proof}
We proceed by induction. First we will prove the statement for $t=0$. In this case, because of the way $\chs$ is initialized  (see line \ref{line:oneiginitialization} of $\oneig$),  $s_0=\hess_{nn}+w_0$ for $w_0 \sim D(0, \pretolr )$. So, by definition, $|s_0|\leq \|H\|+\pretolr$, and by Observation \ref{obs:etasaresmall} we have $s_0\in D(0, C\|H\|)$  for $C=1.02$. It follows, by Observation \ref{obs:accuracy}, that $\tau_0$ satisfies the inequalities in (\ref{eq:guaranteefori}). Therefore $$\tau_0 \leq 1.1 \ds{s_0}{H} \leq 1.1\cdot 2.02 \|H\| \leq 2.3 \|H\|$$ which we record for later use.  

Now take $k\leq t$ and assume that (\ref{eq:guaranteefori}) holds for $i=0, \dots,k$, we will then show that it also holds for $k+1$. First note that by the assumption that $\oneig$ does not terminate in the first $t$ while loops, we have that $  \tau_{i+1} \leq .66 \tau_i$ and $.9\fward \leq \tau_i$ for all $i=0, \dots, k$. Hence, by construction of the sequence $s_0, s_1, \dots $, for any $s\in \chn_{s_k, \tau_k}$ we can obtain
\begin{align*}
\big|s\big| & \leq |s_0|+ |s_1-s_0|+ \cdots +|s_k-s_{k-1}| +|s-s_{k}|
\\ & \leq |s_0|+ \tau_0+|w_1|+ \cdots +\tau_{k}+|w_{k+1}|  && \text{since }\, s_{i+1}\in \chn_{s_{i}, \tau_{i}}, \, s\in \chn_{s_k, \tau_k}
\\ & \leq |s_0| + 1.3(\tau_0+ \cdots + \tau_{k}) && \tau_i \geq 0.9 \fward\,  \text{ and }\, \pretolr \leq  \frac{\fward}{5}
\\ & \leq |s_0| + 1.3 \cdot 2.3 \|\hess\| (1+ 0.66 + 0.66^2+\cdots ) && \tau_{i+1}\leq 0.66^i \tau_0 \leq 0.66^i 2.3 \|\hess\|
\\ & \leq |s_0| + 8.8\|\hess\|
\\ & \leq 10 \|\hess\| && |s_0| \leq 1.02 \|\hess\|. 
\end{align*}
This proves that $\chn_{s_k,\tau_k }\subset D(0, 10\|H\|)$. So, when $k\leq t-1$ we get get that $s_{k+1}\in D(0, 10\|H\|)$, and because we also know that $\ds{s_{k+1}}{H}\geq \tol$, we can apply Observation \ref{obs:accuracy} to show that (\ref{eq:guaranteefori}) holds for $i=k+1$. 
\end{proof}

In the above lemma we assumed that $\oneig$ did not terminate in the first $t$ calls to the while loop, which tacitly assumes that the the flag $\cor$ was set back to $\true$ in each of those loops. We now show that if $\tau_t$ is sufficiently accurate and the elements in $\chn_{s_t, \tau_t}$ are far enough from $\Spec{H}$, then there is a guarantee that in the while loop $t+1$ the flag $\cor$ will be set back to $\true$. 

\begin{lemma}[Guaranteeing $\cor=\true$]
\label{lem:guaranteeingcor}
Assume that $\ds{s_i}{H}$ for $i=1, \dots, t$  and moreover that  each $s\in \chn_{s_t, \tau_t}$ satisfies that $\ds{s}{H}\geq \tol$. Then
$$\min_{s\in \chn_{s_t, \tau_t}} \dispec(s, H, m) \leq .66 \tau_t,$$
where $m$ is defined as in line \ref{line:mandeta} of $\oneig$. 
\end{lemma}

\begin{proof}
Because $\tau_t$ satisfies (\ref{eq:guaranteefori}) we know that there is at least one eigenvalue of $H$ in $\calA_{s_t, \tau_t}$. By Observation \ref{obs:scalereduction} there is at least one $s\in \chn_{s_{t+1}, \tau_{t+1}}$ for which $\ds{s}{H}\leq 0.6 \tau_t$. Moreover,  by assumption, for such $s$ we know that $\ds{s}{H}\geq \tol$, and by Lemma \ref{lem:tauaccuracy} we also know that $s\in D(0, 10\|H\|)$. Hence Observation \ref{obs:accuracy} implies that
$$\dispec(H, s, m) \leq 1.1 \ds{s}{H} \leq 0.66 \tau_t,$$
as we wanted to show. 
\end{proof}

Lemmas \ref{lem:tauaccuracy} and  \ref{lem:guaranteeingcor} imply that as long as all of the values of $\chs$ and $\chs^{(j)}$ for $j=1, \dots, 6$ satisfy that $\ds{\chs}{H}\geq \tol$ and $\ds{\chs^{(j)}}{H}\geq \tol$, we will have accurate $\tau_i$ and the flag $\cor$ will always be set back to $\true$. We can now conclude the proof. 

\paragraph{Probability of success.} Take $t= \lceil 2 \log(\scale/5\fward)\rceil$,  which is set so that $ 4.6\cdot 0.66^t /0.9 \leq \fward/\Sigma. $ 

For $i=1, \dots, t$ and $j=1, \dots, 6$ let $s_i^{(j)}$ be the value acquired by the variable $\chs^{(j)}$ during the while loop $i$. Using Lemma \ref{lem:fixguarantee1} and taking a union bound we have that the probability that
$$\ds{s_0}{H}\geq \tol\quad \text{and}\quad \ds{s_i^{(j)}}{H}\geq \tol, \quad \forall i\in [t]\,  \forall j\in [6]$$
is at least $1-(6t+1)(\tol/\pretol)^2$. And from the above discussion we know that under this event $\oneig$ will not terminate in the first $t$ while loops with $\cor=\false$, and moreover $\tau_0\leq 2.3\|H\|$  and $\tau_{i+1}\leq .66 \tau_i$. Therefore, because $\|H\|\leq 2\scale$ and the way we have chosen $t$,
$$\tau_t \leq 0.66^{t} \tau_0\leq   0.66^{t} 2.3 \|\hess\|  \leq 0.66^{t} \cdot 4.6 \scale \leq    0.9 \fward.$$
This ensures that the algorithm terminates with $\cor=\true$ sometime in the first $t$ while loops with probability at least $1-(6t+1)(\tolr/\pretolr)^2$. Moreover, when it terminates, say at time $t_0$, we are guaranteed that $\ds{s_{t_0}}{H}\geq \tol$, and because $\tau_{t_0}$ is accurate we have that
$$.9 \ds{s_{t_0}}{H}\leq \tau_{t_0}\leq .9\fward,$$
which implies that $\ds{s_{t_0}}{H}\leq \fward$. 

On the other hand 
$$(6t+1)(\tolr/\pretolr)^2 =(6\lceil 2 \log(\scale/5\fward)\rceil+1)(\tolr/\pretolr)^2\leq 12 \log(3\scale/10\fward) (\tolr/\pretolr)^2 = \varphi,$$
that is, the failure probability is upper bounded by $\varphi$.

\paragraph{Running time.} Finally, we give an upper bound on the running time. First note that each iteration of the while loop calls $\dispec$ six times, draws one sample from $\unif(D(0, \pretolr))$,  and at most other 16 arithmetic operations are done. Since, in the successful event, there are at most $\lceil 2 \log(\scale/5\fward)\rceil$ while loops, this gives us the count of 
$$\lceil 2 \log(\scale/5\fward)\rceil (T_{\dispec}(n, m)+ \cd + 16). $$
Before the while loops $\dispec$ is called once, and other than that at most $O(1)$ operations are done. This yields the advertised result.

%% file: decoupling.tex
\section{Decoupling via Inverse Iteration}
\label{sec:decoupling}

The following results are the basis of the subroutine we use to decouple a Hessenberg matrix once a forward approximate eigenvalue of $H$ is obtained. 

\begin{lemma}[Decoupling in Exact Arithmetic]
\label{lem:exactdecoupling}
Let $s\in \bC$ and $H\in \bC^{n\times n}$ be a Hessenberg matrix. Consider the sequence given by $H_0 := H$ and $H_{{\ell}+1} := R_\ell Q_\ell+s$ for $[Q_\ell, R_\ell] := \qr(H_{\ell}-s)$. Then, for any $m\geq 1$ there is some $1\leq \ell \leq m$ for which
 \begin{equation}
 \label{eq:exactdecoupling}
 |(H_{\ell})_{n, n-1}| \leq  \frac{\kappa_V(H)^{\frac{1}{m}}\ds{s}{H}}{\P\Big[|Z_H-s| =   \ds{s}{H}\Big]^{\frac{1}{2m}}}.
 \end{equation}
\end{lemma}

\begin{proof}
Because by definition: $R_\ell$ is upper triangular, all the entries of $Q_\ell$ are bounded by 1, and  $H_{\ell+1} = R_\ell Q_\ell+s$,   we know that 
    \begin{equation}
        \label{eq:upperboundonentry}
        |(H_{\ell+1})_{n, n-1}| \le |(R_{\ell})_{n,n}|.
    \end{equation}
    On the other hand 
    \begin{align}
    |(R_{0})_{n, n} \cdots (R_{m-1})_{n, n}|^{\frac{1}{m}}& = \|e_n^* (H-s)^{-m}\|^{-\frac{1}{m}} \nonumber  && \text{\cite[Lemma 3.3]{banks2022II}} \\ 
        & \le \frac{\kappa_V(H)^{\frac{1}{m}}}{\E\left[|Z_H-s|^{-2m}\right]^{\frac{1}{2m}}} \nonumber && \text{Lemma \ref{lem:spectral-measure-apx}} \\ \label{eq:boundonprodofRs}
        &\le \frac{\kappa_V(H)^{\frac{1}{m}}\ds{s}{H}}{\P\Big[|Z_H-s| =   \ds{s}{H}\Big]^{\frac{1}{2m}}}. 
    \end{align}
    So, combining  (\ref{eq:upperboundonentry}) and (\ref{eq:boundonprodofRs}) we get that  (\ref{eq:exactdecoupling}) holds for some $1\leq \ell\leq m$.   
\end{proof}

Using the forward error guarantees for $\iqr$ given in Lemma \ref{lem:multiiqrstability} we can easily get a finite arithmetic version of the above result.

\begin{lemma}[Decoupling in Finite Arithmetic]
\label{lem:decoupling}
    Let $H\in \bC^{n\times n}$ be a Hessenberg matrix and $s\in D(0, C \|H\|)$.    For every $\ell$ define $\ax{H_\ell} = \iqr(H,(z-s)^\ell)$ . Then, for each $m \ge 1$, if 
    \begin{equation}
    \label{assump:decoupling}
        \mach_{}  \leq 
       \min_{\ell\in [m]} \mach_{\iqr}\big(n,\ell,  \|H\|, \kappa_V(H), \dist(s,\Spec{H})\big)
    \end{equation}
    there is some $\ell \in [m]$ for which 
    \begin{align*}
      |(\ax{H_{\ell}})_{n, n-1}| \le   \frac{\kappa_V(H)^{\frac{1}{m}}\ds{s}{H}}{\P\Big[|Z_H-s| =   \ds{s}{H}\Big]^{\frac{1}{2m}}}  +  32\kappa_V(H) \|H\|\left(\frac{(2 + 2C)\|H\|}{\dist(s,\Spec{H})}\right)^\ell n^{1/2}\muqr(n)\mach.
    \end{align*}
    
\end{lemma}

\begin{proof}
Let $H_0, \dots, H_m$ be as in the statement of Lemma \ref{lem:exactdecoupling}, and let $\ell\in [m]$ be such that (\ref{eq:exactdecoupling}) holds. Now,  (\ref{assump:decoupling}) ensures that we can apply Lemma \ref{lem:multiiqrstability} for the  $\ell$ we have specified, yielding
    $$\left|(H_\ell)_{n, n-1}-  (\ax{H_\ell})_{n, n-1} \right| \leq \left\|H_\ell - \ax{H_\ell} \right\|_F \leq 32\kappa_V(H) \|H\|\left(\frac{(2 + 2C)\|H\|}{\dist(s,\Spec{H})}\right)^\ell n^{1/2}\muqr(n)\mach.$$
    Combining this with (\ref{eq:exactdecoupling}) the advertised bound follows. 
\end{proof}

\subsection{Analysis of $\decouple$}

In view of the above results we define the subroutine $\decouple$ as follows. 
\bigskip

\begin{boxedminipage}{\textwidth}
$$\decouple$$
    \textbf{Input:} Hessenberg $\hess\in \bC^{n\times n}$, $\ax{\lambda}\in\bC$, and decoupling parameter $\wacc>0$     \\
    \textbf{Output:} $\next{H}\in \bC^{n\times n}$ Hessenberg matrix  \\
    \textbf{Requires:} $0< \ds{\ax{\lambda}}{H} \leq \omega/2$  \\
    \textbf{Ensures:} $|\next{H}_{n, n-1}| \leq \wacc$ and there exists a unitary $Q$ with $\|\hat{H}- Q^* HQ\|\leq 3.5 m \|H\|\muqr(n) \mach$, for $m$ defined as in the statement of Proposition \ref{prop:decouple}
\begin{enumerate}
    \item $\hat{H}\gets H$
    \item \label{line:decwhileloop} \textbf{While} $|\next{H}_{n, n-1}| > \omega$
    \begin{enumerate}[label= (\roman*)]
        \item $\next{\hess} \gets \iqr(\next{\hess}, z-\ax{\lambda})$
\end{enumerate} 
    \item Output $\hat{H}$
\end{enumerate}
\end{boxedminipage}

\begin{proposition}[Guarantees for $\decouple$]
\label{prop:decouple}
Assume that the requirements of $\decouple$ are satisfied,  that $H$ is diagonalizable, and that $d:=\ds{\ax{\lambda}}{H}$ and $p:=\P\big[|Z_H-\ax{\lambda}| =   d\big]$ are positive.  If 
\begin{align}
\label{assum:decouple}    \mach & \leq \mach_{\decouple}\big(n, \|H\|, \kappa_V(H), p, d\big) 
    \\ & := \frac{\mach_{\iqr}(n,m,\|H\|,\kappa_V(H),d)\omega}{16\cdot 5^m\cdot   n^{1/2} \|H\|},    \nonumber
\end{align}
for $m= \left\lceil \frac{\log(\kappa_V(H)^2/p)}{2 \log(3\omega/4d)} \right\rceil$, then $\decouple$ satisfies its guarantees and halts after at most $m$ calls to $\iqr$. Hence, it runs in at most
\begin{align*}
T_{\decouple}(n, \kappa_V(H), p, d)   := m T_{\iqr}(n, m)  = O\left( \log(\kappa_V(H)/p)^2 n^2 \right)
\end{align*}
arithmetic operations. 
\end{proposition}

\begin{proof}
First, if $\wcc\geq \|H\|$ the while loop in line \ref{line:decwhileloop} terminates immediately and $\decouple$ satisfies its guarantees after one arithmetic operation. Hence, we can assume $\wcc\leq \|H\|$, which combined with the assumption $d \leq \wcc/2$ gives $d\leq \|H\|/2$ and $\ax{\lambda} \in D(0, 1.5 \|H\|)$ . 

Now, for every $\ell$ define  $\ax{H_\ell} := \iqr\big(H, (z-\ax{\lambda})^\ell\big)$,  and note that (\ref{assum:decouple}) implies that 
$$\mach_{}  \leq \mach_{\iqr}(n, m, \|H\|, \kappa_V(H), d)=  \min_{\ell\in [m]} \mach_{\iqr}\big(n,\ell,  \|H\|, \kappa_V(H),  d\big),$$
where the last equality follows from $d\leq \|H\|/2$. Therefore, we can apply Lemma \ref{lem:decoupling} to get that there is some $\ell\in [m]$ for which
\begin{align*}
      |(\ax{H_{\ell}})_{n, n-1}| & \le  \left(\frac{\kappa_V(H)^2}{p}\right)^{\frac{1}{2m}}d  +  32\kappa_V(H) \|H\|\left(\frac{5\|H\|}{d}\right)^\ell n^{1/2}\muqr(n)\mach.
    \end{align*}
Now, by our choice of $m$ we have that 
$$\left(\frac{\kappa_V(H)^2}{p}\right)^{\frac{1}{2m}}d \leq \frac{3\omega}{4},$$
and by (\ref{assum:decouple}), because $\ell\leq m$ and $\wcc\leq \|H\|$, we have that 
$$ 32\kappa_V(H) \|H\|\left(\frac{5\|H\|}{d}\right)^\ell n^{1/2}\muqr(n)\mach\leq \frac{\wcc}{4}.$$
Combining the above inequalities we get that $|(\ax{H_\ell})_{n, n-1}| \leq \omega$ as we wanted to show. To prove the remaining claim use again that $\ax{\lambda}\in D(0, C \|H\|)$  for $C=1.5$, and apply Lemma \ref{lem:iqr-multi-backward-guarantees} to get that there is a unitary $Q$ for which
$$\|\ax{H_\ell}- Q^* H Q\|\leq 1.4 \ell (1 + C)\|H\|\muqr(n)\mach \leq 3.5 m \|H\|\muqr(n)\mach,$$
as we wanted to show. 
\end{proof}

%% file: randhess.tex
\section{Randomized Hessenberg Form}
\label{sec:rhess}

Some of the most common and well understood subroutines in numerical linear algebra are those used to put an arbitrary matrix $M\in \bC^{n\times n}$ into a Hessenberg form $H$ (e.g. see \cite{demmel1997applied, higham2002accuracy}). The only reason why we have decided to include this section in the present paper, is that we were not able to find in the literature a  rigorous result about the effect of randomizing the Hessenberg form $H$ that could allow us to conclude an explicit probabilistic lower bound on $\min_{\lambda\in \Spec(H)} \P[Z_H=\lambda]$. Here, in our analysis we assume access to a deterministic algorithm that uses Householder reflectors to obtain the Hessenberg form (see Definition \ref{def:buh} below for details), and to a random unit vector generator satisfying the assumptions from Definition \ref{def:usampler} above.

\subsection{Householder Reflectors}

Computing Householder reflectors is essential to many numerical linear algebra algorithms and a thorough analysis of the numerical errors involved can be found in \cite[Section 19.3]{higham2002accuracy}. In short, Householder reflectors are matrices $P\in \bC^{n\times n}$ of the form $P=I-\beta v v^*$ with $v\in \bC^{n} \setminus \{0\}$ and $\beta:= \frac{2}{v^* v}. $\footnote{It is easy to see that $P$ is a reflection over the hyperplane $\{v\}^\perp$.} In practice, given  $v$, instead of computing $P$ it is more convenient to simply store $v$, which for any vector $x$ allows to compute $P x$ by just computing $x-\beta (v^* x) v$ and this takes 
$$T_{\mathsf{hous}}(n) = O(n)$$
 arithmetic operations.

With this in mind, given $x, v \in \bC^{n}$ we will use $\hous{v, x}$ to denote the finite arithmetic computation of $Px$ following the procedure outlined above. Similarly, given $A\in \bC^{n\times n}$ we will use $\hous{v, A}$ to denote the finite arithmetic computation of $PA$, where the $i$-th column of $PA$ is computed as $\hous{v, A^{(i)}}$ where $A^{(i)}$ denotes the $i$-th column of $A$.  

In \cite[Lemma 19.2]{higham2002accuracy} it was shown that there exists a small universal constant $\chh$ for which, provided that $\chh n\mach < 1/2$, one has 
\begin{equation}
\label{eq:errorhouseholder}
\hous{v, x} = (P+E)x  \quad \text{for} \quad \|E\|_F \leq 2\chh n \mach, 
\end{equation}
for any $x\in \bC^{n}$. This will be used later in the analysis of $\rhform$. 

\subsection{Hessenberg Form}

The standard way in which a matrix $M\in \bC^{n\times n}$ is put into Hessenberg form using Householder reflectors is by using a \emph{left-to-right} approach, where one generates a sequence of Householder reflectors $P_1, \dots, P_{n-2}$, that ensure that $H:= P_{n-2} \cdots P_1 M P_1 \cdots P_{n-2}$ is Hessenberg, and where each $P_i$ is used to set to zero the entries in \emph{column} $i$ of the working matrix that are \emph{below} the subdiagonal. 

However, since we will be interested in randomizing the relative position of $e_n$ with respect to the eigenbasis of $H$, it will be  convenient to instead use a \emph{bottom-up} approach, and choose each $P_i$ to set to zero the entries in \emph{row} $i$ that are to the \emph{left} of the corresponding subdiagonal. In this way, when acting on the left of the matrix, the $P_i$ leave the $n$-th row of the working matrix invariant and, in particular, we will have  $e_n^* P_i = e_n^*$. Since the left-to-right and bottom-up approaches are essentially equivalent, the results from \cite[Theorem 2]{coise1996backward} and \cite[Section 4.4.6]{demmel1997applied} apply in both situations, and in particular imply the existence of an efficient and backward stable algorithm in the following sense. 

\begin{definition}[Bottom-up Hessenberg Form Algorithm]
\label{def:buh}
A $\ch$-stable bottom-up Hessenberg form algorithm $\hform$, is an algorithm that takes as input a matrix $M \in \bC^{n\times n}$ and outputs a Hessenberg matrix $\hess\in \bC^{n\times n}$ satisfying that there exists a unitary $Q$ with
$$\| H - Q^*M Q\|\leq   \ch \|M\| n^{5/2} \mach  $$
and such that $Q e_n = e_n$.  We say that $\hform$ is efficient if it runs in at most
$$T_{\hform}(n) := \frac{10}{3} n^3 + O(n^2)$$
arithmetic operations. 
\end{definition}

\subsection{Analysis of $\rhform$}

As mentioned above, the only source of randomness for $\hform$ is a random vector uniformly sampled from the complex unit sphere. Our main technical tool for the analysis will be the following standard anti-concentration result, whose proof we defer to the appendix.  

\begin{lemma}[Anti-Concentration for Random Vectors]
\label{lem:anticoncentration}
Let $u\sim \unif(\bS^{n-1}_\bC)$ and $v\in \bC$ with $\|v\|=1$. Then for all $t\in [0, 1]$
$$\P\left[ |u^* v| \leq \frac{t}{\sqrt{n-1}} \right]  \leq  t^2.$$
\end{lemma}

We can now define the algorithm and proof its guarantees. 
\bigskip

\begin{boxedminipage}{\textwidth}
$$\rhform$$
    \textbf{Input:}  $M\in \bC^{n\times n}$    \\
    \textbf{Output:}  $H \in \bC^{n\times n}$ \\
    \textbf{Requires:} $\Lambda_\epsilon (M)$ is $\zeta$-shattered \\
    \textbf{Ensures:} $\hess$ is Hessenberg, $\|\hess - Q^* M Q\|\leq  \crh \|M\|  n^{5/2}\mach $  for some unitary $Q$, $\Lambda_{\epsilon'}(\hess)$ is \\ $\zeta$-shattered for $\epsilon' = \epsilon -\crh \|M\|  n^{5/2}\mach  $. Moreover, for any $t$, with probability at least $1-nt^2$ it holds that    $\P[Z_{\hess} = \lambda] \geq \left(\frac{\epsilon' t}{n^{3/2} \shat} \right)^2$  for all $\lambda \in \Spec{H}$
\begin{enumerate}
    \item $u \sim \unif(\bS^{n-1}_{\bC})$
    \item \label{line:firstconjugatedmatrix} $H \gets \hous{u-e_n, M}$ 
    \item \label{line:secconjugatedmatrix} $H \gets \hous{u-e_n, H^*}^*$
    \item $H \gets \hform\big(H\big)$
\end{enumerate} 
\end{boxedminipage}
\bigskip

\begin{proposition}[Guarantees for randomized Hessenberg form] 
\label{prop:rhformguarantees}
Assume that 
\begin{equation}
\label{assum:RHess}
\mach \leq \mach_{\rhform}(n) := \frac{1}{20 \chh n^{3/2}}.
\end{equation}
Then, $\rhform$ satisfies its guarantees for $\crh = 3(\ch+\chh)$ and can be instantiated using at most 
$$T_{\rhform}(n) :=  T_{\hform}(n)+ 2nT_{\mathsf{hous}}(n)+\cu n = O(n^3). $$
arithmetic operations. 
\end{proposition}

\begin{proof}
The case $n=1$ is trivial so we assume $n\geq 2$.  Let $H$ be the output of $\rhform(M)$,  $A_1$ and $A_2$ be the matrices computed in lines \ref{line:firstconjugatedmatrix} and \ref{line:secconjugatedmatrix} of $\rhform$, $P=I-\beta vv^*$ for $v=u-e_n$ (and $\beta=\frac{2}{v^*v}$),   and define $E_1:= A_1-PM$ and $E_2 :=A_2-A_1P$.  From (\ref{eq:errorhouseholder}) it is easy to see that 
$$\|E_1\| \leq 2 \chh \|M\| n^{3/2} \mach \quad \text{and} \quad \|E_2\| \leq 2 \chh \|A_1\| n^{3/2} \mach.$$
Using the first inequality and (\ref{assum:RHess}) we get that $\|A_1\|\leq \|E_1\|+\|M\| \leq 1.1 \|M\|$. Then, combining this with the second inequality we get $\|E_2\| \leq 2.2 \chh \|M\| n^{3/2} \mach$. Hence 
\begin{equation}
\label{eq:rhess1}
\|A_2-PMP\| \leq \|A_2-A_1P\|+ \|A_1P-PMP\| = \|E_1\|+\|E_2\|\leq 4.2 \chh \|M\| n^{3/2} \mach. 
\end{equation}
Again because of (\ref{assum:RHess}) the above inequality implies that $\|A_2\| \leq 1.3 \|M\|$. So, by Definition \ref{def:buh} we get that $\|H-Q^* A_2 Q\|\leq 1.3 \ch \|M\| n^{5/2}\mach,$ for some unitary $Q$ satisfying $Qe_n = e_n$, which combined with (\ref{eq:rhess1}) yields 
$$\|H-Q^*PMPQ\| \leq (1.3\ch n^{5/2}+4.2 \chh n^{3/2})\|M\|\mach \leq \crh \|M\| n^{5/2}\mach,$$
proving the first claim. Now, because $\Lambda_\epsilon(M)$ is $\shat$-shattered, the above inequality and Lemma \ref{lem:decrementeps} imply that $\Lambda_{\epsilon'}(H)$ is $\shat$-shattered for $\epsilon'=\epsilon -\crh\|M\|n^{5/2} \mach$. 

It remains to prove the anti-concentration statement for $Z_H$. To do this let $E\in \bC^{n\times n}$ be such that $H = Q^*P(M+E)PQ$, and let $M+E = VDV^{-1}$ with $D=\mathrm{diag}(\lambda_1, \dots, \lambda_n)$ and $V$ chosen so that $\|V\|=\|V^{-1}\| = \sqrt{\kappa_V(M+E)}$. Now note that $Q^*PV$ is an eigenvector matrix for $H$, and  because $P$ and $Q$ are unitary $\|Q^*PV\| = \| V\| =\sqrt{\kappa_V(M+E)} = \sqrt{\kappa_V(H)}$. So
\begin{align*}
\P[Z_\hess  = \lambda_i] & = \frac{|e_n^* Q^* P V e_i|^2}{\|e_n^* Q^* P V\|^2} && \text{definition of } \P[Z_{\hess} = \lambda_i] 
\\ & = \frac{|e_n^*  P V e_i |^2}{\|e_n^* P V\|^2}     && e_n^* Q^* = e_n^*
\\ & = \frac{|u^* V e_i|^2 }{\|u^* V\|^2} &&  u  =  P e_n\text{ by definition of }P.
\end{align*}
To simplify notation define  $v_i:= \frac{V e_i}{\|Ve_i\|}$. We then have 
\begin{align*}
    \frac{|u^* V e_i|^2 }{\|u^* V\|^2} & =  \frac{|u^* v_i|^2 \|Ve_i\|^2 }{\|u^* V\|^2}
    \\ & \geq \frac{|u^* v_i|^2 }{\| V\|^2\|V^{-1}\|^2} && \|Ve_i\| \geq \frac{1}{\|V^{-1}\|} \text{ and } \|u^* V\| \leq \|V\|
    \\ &= \frac{|u^* v_i|^2 }{\kappa_V(H)^2} && \kappa_V(M+E) = \kappa_V(H)
    \\ & \geq \left(\frac{\epsilon'|u^*v_i|}{n\shat}\right)^2 && \Lambda_{\epsilon'}(H) \text{ is } \shat\text{-shattered  and  Lemma \ref{lem:kappavfromshattering}}. 
\end{align*}
Now, because $\|v_i\|=1$, we can apply Lemma \ref{lem:anticoncentration} to get that for any $t\geq 0$
$$\P\left[|u^* v_i| \geq \frac{t}{\sqrt{n-1}}\right] \geq 1- t^2.$$
Which, in conjunction with the above gives that 
$$\P[Z_H=\lambda_i] \geq \left(\frac{\epsilon' t}{n^{3/2} \shat} \right)^2, $$
with probability at least $1-t^2$. The advertised claim then follows from taking a union bound over all $v_i$. The claim about the running tie follows trivially.   
 
\end{proof}

%% file: eig.tex
\section{The Main Algorithm}
\label{sec:eig}

So far, all but one of the subroutines required to define $\eig$ have been discussed. The remaining subroutine that will be needed is the one used for deflation, denoted here by $\deflate(H, \omega)$, which on a Hessenberg input $H\in \bC^{n\times n}$ sets to zero any of the $n-1$ subdiagonals of $H$ that are less or equal (in absolute value) to $\omega$, and returns the diagonal blocks $H_1, H_2, \dots$ of the resulting matrix.  

We are now ready to define the main algorithm and prove its guarantees. Note that $n$ refers to the dimension of the original input matrix, which is used to set parameters throughout the recursive calls to $\findritz$.
\bigskip

\begin{boxedminipage}{\textwidth}
$$\findritz$$
    \textbf{Input:}  Complex matrix $M$,  accuracy $\bck$,  failure probability tolerance $\phi$   \\
    \textbf{Global Data:} Dimension $n$, norm estimate $\scale$,  pseudospectral parameter $\epsilon$, shattering parameter $\shat$\\ 
    \textbf{Output:} A multiset $\Lambda \subset \mathbb{C} $ \\
    \textbf{Ensures:} $\Lambda$ is the spectrum of a matrix $\ax{M}$ with $\|M- \ax{M}\| \leq \bck $. 
\begin{enumerate}
 \item \label{line:eigparameters}   $\absacc \gets \frac{\delta \scale}{2}, \,  \wcc\gets \frac{\epsilon \wedge \absacc}{3 n} $, \,  $\fward\gets \frac{\wcc}{20} , \, p\gets \frac{ \phi \epsilon^2}{2 n^{5} \shat^2} $, \,$\wfail\gets \frac{\phi}{2n}$,\, $\cor \gets \false$
    \item \label{line:eigrhess} $\hess \gets \rhform(M)$    
    \item \label{line:eigwhile} \textbf{While} $\cor = \false$ \\
    $[\ax{\lambda}, \cor] \gets \oneig(\hess, \fward,\wfail, p) $
    
    \item $H \gets \decouple(H, \ax{\lambda}, \omega)$
 
    \item $[M_1, M_2, \dots ] \gets \deflate(\hess, \wcc)$
 
    \item $\Lambda \gets   \bigsqcup_i \findritz\big(M_i, \delta  ,  \phi \big) $
\end{enumerate} 
\end{boxedminipage}
\bigskip

\begin{theorem}
\label{thm:mainquantitative}
Let $M$ be the input matrix and $\delta \in (0, 1)$. Let  $\beta, \absacc$ and  $p$ be as in line \ref{line:eigparameters} of $\eig$.   Assume that the global data satisfies $n=\dim(M)$, $\scale/2 \leq \|M\|\leq \scale$, and that $\Lambda_{2\epsilon}(M)$ is $\shat$-shattered. If
\begin{align}
\label{assum:eig}
\mach & \leq \mach_{\eig}(n, \scale, \epsilon, \shat, \delta, \phi) 
\\ & :=  \frac{\epsilon}{6\cdot 10^3 (\chh\vee \ch\vee \croot)\muqr(n) n \shat } \left( \frac{\tolr}{44\scale}\right)^{2m_1}
\nonumber
\end{align}
where 
\begin{align}
&m_1=\left \lceil 12\log(n\shat/\epsilon)+ 6\log(1/p)\right\rceil =O\big(\log(n\shat/\epsilon\phi)\big) \nonumber
\\ \label{eq:etaandm} \text{and} \quad  &\tolr= \left( \frac{\epsilon\wedge \absacc}{300n}\right) \left( \frac{\phi}{24 n\log(18\scale n/(\epsilon\wedge \absacc))} \right)^{1/2} = O\left( \frac{(\epsilon\wedge \absacc)\phi^{1/2} }{n \log(\scale n/(\epsilon\wedge \absacc))^{1/2}}\right) 
\end{align}
Then, with probability at least $1 - \phi$ $\eig$ satisfies its guarantees, and in this event runs in at most
\begin{align*}
T_{\eig}(n, \epsilon, \shat, \delta)  &=  (n - 1)(T_{\rhform}(n)+ T_{\oneig}(n, \scale, \epsilon, \shat, p, \fward)+ T_{\decouple}(n, \shat n/\epsilon, p , \beta) ) 
\\  &=  O\Big( n^4 + n^3 \log(\shat n/\epsilon\phi)\big( \log(\scale n/(\epsilon\wedge \absacc))+ \log(\shat n/\epsilon \phi) \big) 
\\  & \quad \quad   + \log  (\scale n/(\epsilon\wedge \absacc))\log(\shat n/\epsilon \phi) \log \log(\shat n/\epsilon \phi)\Big)    
\end{align*}
arithmetic operations. 
\end{theorem}

\subsection{Preservation of the Norm and Pseudospectral Parameters}

Before delving into the analysis of $\eig$ we will show that the global data provides valuable information throughout the execution of the algorithm. The first observation here is that the only subroutine of $\eig$ that accesses  the global data is $\oneig$, so, to ensure correctness, the only requirements  regarding the global data that need to be fulfilled are the ones ensured by the following lemma. 

\begin{lemma}
\label{lem:controloftheparameters}
Suppose that the assumptions of Theorem \ref{thm:mainquantitative} are satisfied and let $H'$ and $M'$ be any values acquired by the variables $H$ and $M$. If every while loop (line \ref{line:eigwhile}) involved in the production of $H'$ and $M'$ ended with $\oneig$ being terminated successfully, then:  
\begin{enumerate}[label=\roman*)]
    \item $\frac{\wcc}{2}\leq \min\{ \|H'\|, \|M'\|\} \leq \max\{\|H'\|, \|M'\|\} \leq 2\Sigma$. 
    \item $\Lambda_{\epsilon}(H')$ and $\Lambda_{\epsilon}(M')$ are $\shat$-shattered. 
\end{enumerate}
\end{lemma}

To prove the above lemma we will need the following results to control the pseudospectral parameters after each deflation step (controlling the norm  after deflation is trivial). 

\begin{lemma}[Lemma 5.9 in \cite{banks2020pseudospectral}] \label{lem:compress-shattered} 
 Suppose $P$ is a spectral projector of $M\in\bC^{n\times n}$ of rank $r\leq n$. Let $S\in \bC^{n\times r}$ be such that $S^*S=I_r$ and that its columns span the same space as the columns of $P$. Then for every $\epsilon>0$,
    $$\Lambda_\epsilon(S^*MS)\subset \Lambda_\epsilon(M).$$
    Alternatively, the same pseudospectral inclusion holds if again $S^*S=I_r$ and, instead, the columns of $S$ span the same space as the rows of $P$. 
\end{lemma}

Combining Lemmas \ref{lem:decrementeps} and \ref{lem:compress-shattered} we can  show the following. 

\begin{lemma}[Pseudospectrum After Deflation]
\label{lem:deflationpseudospectrum}
Let $H\in \bC^{n\times n}$ be a Hessenberg matrix and and $1\leq r\leq n-1$ . Let $H_{-}$ and $H_{+}$ be its upper-left and lower-right $r\times r$ and $(n-r)\times (n-r)$ corners respectively. If $|H_{r+1, r}|\leq \epsilon'$ then 
$$\Lambda_{\epsilon-\epsilon'} (H_-) \cup \Lambda_{\epsilon-\epsilon'}(H_+) \subset \Lambda_\epsilon(H).$$
\end{lemma}

\begin{proof}
Let $H_0$ be the matrix obtained by zeroing out the $(r+1, r)$ entry of $H$. By Lemma \ref{lem:decrementeps} and the assumption $|H_{r+1, r}|\leq \epsilon'$ we get $\Lambda_{\epsilon -\epsilon'}(H_0)\subset \Lambda_\epsilon(H)$.  

We will begin by showing that $\Lambda_{\epsilon-\epsilon'}(H_+)\subset \Lambda (H_0)$.  Let $w\in \bC^r$ be any left eigenvector of $H_+$ and note that, since $H_0$ is block upper triangular, $0_{n-r}\oplus w \in \bC^{n\times n}$ is a left eigenvector of $H_0$. Hence, there is a spectral projector $P$ of $H_0$ for which its left eigenvectors (equivalently its rows) span the space $\text{span}\{e_{n-r+1}, \dots, e_n\}$. Hence the span of the columns of the $n\times (n-r)$ matrix  $$S= \left( \begin{array}{cc}
     0  \\
     I_{n-r} 
\end{array} \right)$$ 
coincides the span of the rows of $P$. So, by Lemma \ref{lem:compress-shattered}, $\Lambda_{\epsilon-\epsilon'} (H_+) = \Lambda_{\epsilon-\epsilon'}(SHS^*) \subset \Lambda_{\epsilon-\epsilon'}(H_0)$.  

The proof that $\Lambda_{\epsilon-\epsilon'}(H_-) \subset \Lambda_{\epsilon-\epsilon'}(H_0)$ is very similar, with the sole difference that this time one should look at the right eigenvectors of $H_-$, and work with columns (rather than rows) of the spectral projector. 
\end{proof}

We can now proceed to the proof of the lemma. 

\begin{proof}[Proof of Lemma \ref{lem:controloftheparameters}]
First note that in each call to $\eig$ the working matrix gets modified exactly once by each of the subroutines $\rhform$, $\decouple$ and $\deflate$. So, there is a sequence of the form 
$$M=M_1, F_1, F_1', M_2, F_2, F_2' \dots$$ that ends in $H'$ (respectively $M'$), and such that  $F_i =\rhform(M_i),F_i'=\decouple(F_i, , \ax{\lambda}_i, \omega)$ and $M_{i+1}$ is one of the matrices in the output of $\deflate(F_i')$. Moreover, by the assumption that $\oneig$ terminated successfully at the end of each while loop, we have that
\begin{equation}
\label{eq:distassumption}
\tolr \leq \ds{\ax{\lambda}_i}{F_i}\leq \beta.
\end{equation}
We will show by induction that for every $i\leq n$ the pseudospectra $\Lambda_{2\epsilon-\epsilon_{i, 0}}(M_i), \Lambda_{2\epsilon-\epsilon_{i, 1}}(F_i)$ and $\Lambda_{2\epsilon-\epsilon_{i, 2}}(F_i')$ are $\shat$-shattered, where 
$$\epsilon_{i, j} := (3(i-1)+j)\wcc= \frac{3(i-1)+j}{3n}(\absacc \wedge \epsilon),$$
and that $\|M_i\|\leq \scale + \epsilon_{i, 0}, \|F_i\|\leq \scale + \epsilon_{i, 1}$ and $\|F_i'\|\leq \scale+\epsilon_{i, 2}$. Note that in particular this will imply that $\epsilon$-pseudospectra of the $M_i, F_i$ and $F_i'$ are $\shat$-shattered, and their norms are bounded by $2\scale$ (since $\epsilon\leq \scale$).  

That $M_1=M$ has the advertised pseudospectral and norm properties follows from the assumption about the global data. We can then induct: 
\begin{itemize}
    \item \emph{Effect of $\rhform$.} Assume that $\Lambda_{2\epsilon-\epsilon_{i, 0}}(M_i)$ is $\shat$-shattered and $\|M_i\| \leq \scale +\epsilon_{i, 0}$. Because $F_i =\rhform(M_i)$, and since (\ref{assum:oneig}) implies that
    $$\mach \leq \mach_{\rhform}(n) \leq \mach_{\rhform}(\dim(M_i)),  $$
we can apply Proposition \ref{prop:rhformguarantees} to get that $\Lambda_{2\epsilon-\epsilon_{i, 0}-\epsilon'}(F_i)$ is $\shat$-shattered for 
\begin{align*}
\epsilon' & = \crh \|M_i\| \dim(M_i)^{5/2} \mach 
\\ & \leq 2 \crh \scale n^{5/2} \mach & &\dim(M_i)\leq n,\, \|M_i\|\leq 2\scale
\\ & \leq \wcc && \text{by (\ref{assum:eig})}.
\end{align*}
So, it follows that $\Lambda_{2\epsilon-\epsilon_{i, 1}}(F_1)$ is $\shat$-shattered. And in the  same way we can get $\|F_i\|\leq \scale + \epsilon_{i, 1}$. 
\item \emph{Effect of $\decouple$.} Now assume that $\Lambda_{2\epsilon-\epsilon_{i, 1}}(F_i)$ is $\shat$-shattered and $\|F_i\|\leq \scale + \epsilon_{i, 1}$. Let $p$ and $\fward$ be as in line \ref{line:eigparameters} of $\eig$ and define 
\begin{equation}
\label{eq:defofmtwo}
m_2:= \left\lceil \frac{\log(\shat^2n^2/p \epsilon^2)}{2 \log(3\omega/4\beta)} \right\rceil =  \left\lceil \frac{\log(\shat^2n^2/p \epsilon^2)}{2 \log(15)}  \right\rceil.
\end{equation}
Now, because $m_2 \leq \lceil .3 \log(\shat n/\epsilon)+ .15 \log(1/p) \rceil$, it is clear that $m_2\leq m_1$, and then it is easy to see that (\ref{assum:eig}) implies 
$$\mach \leq \mach_{\decouple}(2\scale, \shat n /\epsilon, p, \tolr).$$
So, because $\|F_i\|\leq 2\scale$ (by assumption), $\kappa_V(F_i)\leq \shat \epsilon/n$ (since $\Lambda_\epsilon(F_i)$ is $\shat$-shattered and by Lemma \ref{lem:kappavfromshattering}), and $\ds{\ax{\lambda}}{F_i}\geq \tolr$ (by the assumption in (\ref{eq:distassumption})), we can apply Proposition \ref{prop:decouple} to get that there exists a unitary matrix $Q$ for which 
\begin{align*}
\|F_i'-Q^* F_iQ\| & \leq 3.5 m_1 \|F_i\|\muqr(n) \mach
\\ & \leq 7 m_1 \scale \muqr(n) \mach  && \|F_i\| \leq 2\scale
\\ & \leq \wcc && \text{by (\ref{assum:eig})}. 
\end{align*}
Then, by Lemma \ref{lem:decrementeps} and the assumption that $\Lambda_{2\epsilon-\epsilon_{i, 1}}(F_i)$ is $\shat$-shattered, it follows that $\Lambda_{2\epsilon-\epsilon_{i, 2}}(F_i')$ is $\shat$-shattered. And because the norm is preserved under unitary conjugation we also get that $\|F_i'\|\leq \scale + \epsilon_{i, 2}$. 
\item \emph{Effect of $\deflate$. } Assume that $\Lambda_{2\epsilon-\epsilon_{i, 2}}(F_i')$ is $\shat$-shattered, and recall that $M_{i+1}$ is an output of $\deflate(F_i', \wcc)$. Then, by Lemma \ref{lem:deflationpseudospectrum} we have that $$\Lambda_{2\epsilon-\epsilon_{i+1, 0}}(M_{i+1})=\Lambda_{2\epsilon-\epsilon_{i, 2} -\wcc}(M_{i+1})\subset \Lambda_{2\epsilon-\epsilon_{i, 2}}(F_i')$$ and hence $\Lambda_{2\epsilon-\epsilon_{i+1, 0}}(M_{i+1})$ is $\shat$-shattered.  Similarly, we can note that $\|M_{i+1}\|\leq \|F_i''\|+\wcc \leq \scale+ \epsilon_{i+1, 0}$, which concludes the induction. 
\end{itemize}
Now, since the depth of the recursion tree of $\eig$ is at most $n$, and we have proven the above claim for any $M_i, F_i, F_i''$ with $i\leq n$, we can conclude that $\Lambda_{\epsilon}(H')$ is $\shat$-shattered (resp. $\Lambda_{\epsilon}(M')$) and $\|H'\|\leq 2\scale$ (resp. $\|M'\|\leq 2\scale$), as we wanted to show.  

Finally, to show that $\wcc/2\leq \|H'\|$ (resp. $\wcc/2 \leq \|M'\|$), first note that $\wcc \leq \|M_i\| $ for every $i$. Indeed, when $i=1$ we can use the assumption $\delta\leq 1$, which yields $ \absacc \leq \|M\|$, and combine this with $ \wcc\leq \absacc$. For $i>1$ note that $M_i$ is an output of $\deflate(F_{i-1}'', \wcc)$, and hence its subdiagonals are guaranteed to have absolute value at least $\wcc$, which implies that $\wcc\leq \|M_i\|$. We can then proceed as above (using slightly stronger bounds) to show that $\|M_i-F_i\|\leq \wcc/2$ and $\|M_i-F_i''\|\leq \wcc/2$. So the proof is concluded.  
\end{proof}

\subsection{Analysis of $\eig$}

We are now ready to prove Theorem \ref{thm:mainquantitative}. For clarity, let us divide the proof in several parts. 

\paragraph{Backward stability.} Assume that $\eig$ terminates and outputs $\Lambda$.  Moreover,  assume that when running $\eig$, at the end of all the while loops from line \ref{line:eigwhile}, the subroutine $\oneig$ terminated successfully (later we will prove that this occurs with probability at least $1-\phi$). 

We will show that $\Lambda$ is the spectrum of a matrix $\ax{M}$ with $\|\ax{M}-M\|\leq \absacc$ (which combined  with the assumption about the global data gives $\|\ax{M}-M\|\leq \delta \|M\|$). To be precise we will show an equivalent statement, namely that $\Lambda$ is the spectrum of a matrix that is at distance at most $\absacc$ from the class of matrices that are unitarily equivalent to $M$. To do this, for the purpose of the analysis, it will be convenient to imagine that during the deflation process (after setting to zero the small subdiagonals) instead of cutting out the blocks on the diagonal and considering them as separate subproblems, one keeps the full $n\times n$ matrix and continues to operate  on the full matrix in the obvious way. With this view point the algorithm terminates when the working matrix becomes an upper triangular matrix, and its diagonal elements are precisely the elements of $\Lambda$. 

 In the proof of Lemma \ref{lem:controloftheparameters} it was shown that the only subroutines that deviate the working matrix from the unitary orbit of the original matrix are $\rhform, \decouple$ and $\deflate$. Moreover, it was shown that when each of these subroutines is applied, the corresponding backward error incurred is at most of size $\wcc$.  So we need only to give an upper bound for the number of times these subroutines are called. To do this  consider $\calT_n(M)$, the recursion tree of $\eig$, where the input matrix $M$ is placed at the root, and then the children of any vertex $v$ are in one-to-one correspondence with the matrices outputted after running $\deflate$ on the matrix associated to $v$. It is  clear from the construction that leaves correspond to matrices of dimension 1, and internal vertices (vertices that are not leaves) correspond to higher dimensional matrices. Now note that for any internal vertex $v$ it holds that the sum of the dimensions of the matrices associated to the children of $v$ equals the dimension of the matrix associated to $v$. Then, by induction on $n$ it follows that $\calT_n(M)$ has at most $n-1$ internal vertices. And, since the relevant subroutines are only called once at times corresponding to internal leaves, we conclude that each of these subroutines was called at most $n-1$ times. Hence, the ultimate deviation from the original unitary equivalence class is at most 
 $$3(n-1)\omega = \frac{3(n-1)}{3n}(\epsilon\wedge \absacc) \leq \absacc,$$ 
as we wanted to show. 

\paragraph{Precision requirements.} To ensure that the precision has been set to be small enough, so that the precision requirements of each subroutine are satisfied throughout the iteration, we will show that
$$\mach_{\eig}(n, \scale, \epsilon, \shat, \delta, \phi, n) \leq \min \{ \mach_{\rhform}(n), \mach_{\oneig}(n, 2\scale, \epsilon, \shat, p, \fward, \varphi), \mach_{\decouple}(n, 2\scale, \shat n/\epsilon, p, \tolr)\}.$$
First, that $\mach_{\eig}(n, \scale, \epsilon, \shat, \delta, \phi, n)\leq \mach_{\rhform}(n)$ is trivial. On the other hand, by definition we have
\begin{align*}
    \mach_{\oneig}(n, \scale, \epsilon, \shat, p, \fward, \phi) & = \mach_{\dispec}\big(n,m_1, 10,2 \scale, n\shat/\epsilon,\tolr\big) && \text{for } m_1, \tolr \text{ as in  (\ref{eq:etaandm})}
    \\ & \geq \frac{\epsilon}{ 6\cdot 10^3\croot   \cdot  \muqr(n)  n  \shat }\left(\frac{\tolr}{44\scale}\right)^{2m_1} && \text{(\ref{assum:oneig}) and  (\ref{assum:comptau})} 
\end{align*}
So from the  (\ref{assum:eig}) it is clear that $ \mach_{\eig}(n, \scale, \epsilon, \shat, \delta, \phi, n)\leq  \mach_{\oneig}(n, 2\scale, \epsilon, \shat, p, \fward, \phi)$.  Finally
\begin{align*}
    \mach_{\decouple}\big(n, 2\scale, \shat n /\epsilon, p, d\big) & = \frac{\mach_{\iqr}(n,m_2,2\scale,\shat n /\epsilon,\tolr)\omega}{16\cdot 5^{m_2}\cdot   n^{1/2} 2\scale} && \text{for }m_2 \text{ as in (\ref{eq:defofmtwo})}
    \\ & = \frac{\omega \epsilon }{16 \cdot 8 \muqr(n) n^{3/2} \shat \cdot 2\scale }\left(\frac{\tolr}{5\cdot 2\scale}\right)^{m_2} && \text{from (\ref{assum:machvsp})} 
\end{align*}
And because $m_2 \leq m_2$, from (\ref{assum:eig}) it is clear that $\mach_{\eig}(n, \scale, \epsilon, \shat, \delta, \phi, n) \leq  \mach_{\decouple}\big(n, 2\scale, \shat n /\epsilon, p, d\big)$.

\paragraph{Probability of success.} Observe that the only randomized subroutines of $\eig$ are $\rhform$ and $\oneig$.  First we will provide a lower bound for the probability that the guarantees of $\rhform$ and $\oneig$ are satisfied every time  these subroutines are called. 

Combining Lemma \ref{lem:controloftheparameters} and Proposition \ref{prop:rhformguarantees} we get that, if $\oneig$ has  succeeded every time it has been called, then for any value $H'$ acquired by the variable $H$ in line \ref{line:eigrhess} of $\eig$ we have for any $t>0$, with probability $1-nt^2$, that
$$\min_{\lambda \in \Spec{H'}} \P[Z_{H'}=\lambda] \geq \left(\frac{\epsilon t}{n^{3/2}\shat}\right)^2. $$
In particular (for $t^2= \phi/2n^2$) we get that with probability at least $1-\phi/2n$ it holds that 
$$\min_{\lambda \in \Spec{H'}} \P[Z_{H'}=\lambda] \geq p$$ 
for $p$ defined as in line \ref{line:eigparameters}. Under this event, and because of Lemma \ref{lem:controloftheparameters} and because the precision is high enough, the requirements of $\oneig$ will be met in line \ref{line:eigwhile}, and hence (for this call) $\oneig$ will succeed with probability at least $1-\varphi =1- \phi/2n$. 

Therefore, every time $\eig$ is called, both $\rhform$ and $\oneig$ will satisfy their guarantees with probability at least $1-\phi/n$. Moreover, from the backward stability proof we know that the recursion tree for $\eig$ has at most $n-1$ internal vertices. Therefore, we can conclude that all the calls to $\rhform$ and $\oneig$ will succeed with probability at least $1-\phi$, as we wanted to show. 

Now, under the assumption that $\rhform$ and $\oneig$ succeed every time, we have that the values of the variables $H$ and $\ax{\lambda}$ that are passed every time to $\decouple$ satisfy the requirements of this subroutine, and by our previous discussion we know that the precision requirements for $\decouple$ are also met. Therefore, we can apply Proposition \ref{prop:decouple} to argue that the matrix $H$ will be decouple in a finite amount of time, and by Lemma \ref{lem:controloftheparameters} we know that the pseudospectral parameters and norm guarantees will also be maintained. 

\paragraph{Running time.} From the above discussion we know that with probability at least $1-\phi$,  $\eig$ terminates successfully and moreover, throughout the algorithm, every call to $\rhform, \oneig$ and $\decouple$ will be successful, and the requirements of these subroutines will always be met. Under this event (recalling that each subroutine is called at most $n-1$ times) by Propositions  \ref{prop:rhformguarantees}, \ref{prop:findone} and \ref{prop:decouple} and using the the running times of the subroutine are monotone in the dimension of the input, we get that the running time of $\eig$ is at most 
$$(n-1)(T_{\rhform}(n)+ T_{\oneig}(n, \scale, \epsilon, \shat, p, \fward)+ T_{\decouple}(n, \shat n/\epsilon, p , \beta) ).$$
The proof is concluded by writing $p, \beta$ and $\tolr$ as a function of $\epsilon, \shat$, $\scale$ and $\delta$, and using the big-$O$ bounds provided in  Propositions  \ref{prop:rhformguarantees}, \ref{prop:findone} and \ref{prop:decouple}.

\subsection{Pseudospectral Shattering and Proof of the Main Result}
\label{sec:shatandain}

Note that Theorem \ref{thm:mainquantitative} assumes that $\eig$ has access to the parameters $\epsilon$ and $\shat$ in the global data, which control both the minimum eigenvalue gap and the eigenvector condition number of the input matrix $M\in \bC^{n\times n}$. In order to ensure that $\eig$ works on every input (without having access to $\epsilon$ and $\shat$), instead of running the algorithm on $M$ we will run it on $M+\gamma G_n$  (for $\gamma = \Theta (\delta)$ and  $G_n$ a normalized complex Ginibre matrix\footnote{That is, an $n\times n$ random matrix with independent centered complex Gaussian entries of variance $1/n$.}), and exploit the following result, whose proof we defer to Appendix  \ref{sec:shat}.\footnote{A version of this result was already proven and used in a similar context in \cite{banks2020pseudospectral}. However, since the notion of \emph{shattered pseudospectrum} from that paper differs from the one used here, we were not able to directly apply the aforementioned result.} 

\begin{lemma}[Shattering]
\label{lem:shattering}
For any $M\in \bC^{n\times n}$ and $\varphi\in (0, 1/2), \gamma\in (0, \|M\|/2)$, we have that, with probability at least $1-\varphi$,  $\Lambda_\epsilon(M+\gamma G_n)$ is $\shat$-shattered for 
$$\shat :=\frac{\varphi^{1/2}\gamma}{2\sqrt{3} n^{3/2}} \quad \text{and} \quad \epsilon := \frac{\gamma^2 \varphi}{180 \sqrt{2} \|M\| \log(1/\varphi) n^{3}}$$
\end{lemma}

The main result of this paper then follows from combining Lemma \ref{lem:shattering} with Theorem \ref{thm:mainquantitative}.

\begin{proof}[Proof of Theorem \ref{thm:main}] Start by recalling the following the well-known tail bound for the norm of a Ginibre matrix  (e.g. see  \cite[Lemma 2.2]{banks2021gaussian}) 
   \begin{equation}
   \label{eq:ginibrenormbound}
   \P[\|G_n\| \geq t]\leq 2\exp\big(-n(t-2\sqrt{2})^2\big), \quad \forall t \geq 2\sqrt{2}.
   \end{equation}
Then, for $W:= 2\sqrt{2}+\frac{1}{n^{1/2}}\log(6/\phi)^{1/2} $ we have that $$\P[\|G_n\|\leq W]\geq 1-\phi/3.$$ 
Then, given a norm estimate $\scale$ satisfying $\scale/2 \leq \|M\|(1\pm \delta/2) \leq \scale$, we will choose  $\gamma:= \frac{\delta \scale}{4W}$, so that
\begin{align*}
\P\left[\gamma \|G_n\|\leq \frac{\delta \|M\|}{2} \right] & = \P\left[ \|G_n\| \leq \frac{2 \|M\| W }{\scale} \right] 
\\ & \geq \P\left[ \|G_n\| \leq W \right] && \scale/2 \leq \|M\|
\\ & \geq 1-\phi/3.
\end{align*}
 Moreover, for this choice of $\gamma$, by Lemma \ref{lem:shattering} we have that, with probability at least $1-\phi/3$, $\Lambda_\epsilon(M+\gamma G_n)$ is $\shat$-shattered for 
$$\shat :=\frac{\phi^{1/2}\gamma}{2\sqrt{6} n^{3/2}} \quad \text{and} \quad \epsilon := \frac{\gamma^2 \phi}{540\sqrt{2}\|M\| \log(1/\phi) n^{3}}.$$
On the other hand, conditioning on $\|G_n\|\leq W$ and  $\Lambda_\epsilon(M+\gamma G_n)$ being $\shat$-shattered, we  have that $\eig(M+\gamma G_n, \delta/2, \phi/3)$ succeeds with probability at least $1-\phi/3$ when using $n, \epsilon, \shat$ and $\Sigma$ as global data  and provided that $\mach$ satisfies (\ref{assum:eig})), in which case the output $\Lambda$ will be a $\delta$-backward approximation of the spectrum of $M$. 

Hence, using a union bound we get that with probability $1-\phi$, $\eig(M+\gamma G_n, \delta/2, \phi/3)$ provides a $\delta$-accurate answer, and from Theorem \ref{thm:mainquantitative} we have that the running time and required bits of precision are as in the statement of of Theorem \ref{thm:main}. 
\end{proof}

%% file: randvector.tex
\section{Anti-concentration for Random Vectors}
\label{sec:appendixproofs}

\begin{proof}[Proof of Lemma \ref{lem:anticoncentration} ]
Because the distribution of $u$ is unitarily invariant and $\|v\|=1$, we have $u^*v=_d u^*e_i=u(i)$\footnote{Given two random variables $X$ and $Y$, we use $X=_d Y$ to denote that they have the same distribution.} for every $i\in [n]$. So, for concreteness we will take $i=1$ and bound $\P[|u(1)|\leq t]$ for any $t\geq 0$. 

Now recall that  if  $X_1, \dots, X_n, Y_1, \dots, Y_n$ are independent \emph{real} standard Gaussians, then $$u=_d \frac{(X_1+iY_1, \dots, X_k+iY_n)}{\sqrt{X_1^2+Y_1^2+\cdots+ X_k^2+ Y_n^2}} $$
and in particular $|u(1)|^2 = \frac{Z_1}{Z_1+Z_2}$ where $Z_1 \sim \chi^2(2)$ and $Z_2\sim \chi^2(2n-2)$ are independent. Then, we use the well known fact  that $\frac{Z_1}{Z_1+Z_2}$ has a $\Beta(1, n-1)$ distribution, and hence its probability density function is given $f_{\Beta(1, n-1)}(s) = (n-1)(1-s)^{n-2}\cdot 1_{\{0\leq s\leq  1\}} $. It follows that, for $t\in [0, 1]$
$$ \P[|u(1)|\leq t] = \P[|u(1)|^2 \leq t^2] =  (n-1)\int_0^{t^2} (1-s)^{n-2} ds  = 1-(1-t^2)^{n-1} \leq  (n-1)t^2, $$
where the last inequality follows from Bernoulli's inequality.  
\end{proof}

%% file: shattering.tex
\section{Pseudospectral Shattering}
\label{sec:shat}

Here we will use $G_n$ to denote a normalized complex Ginibre matrix. That a  perturbation by $\gamma G_n$ leads to shattering of the pseudospectrum of the perturbed matrix, follows easily from the following lemmas about $\gap$ and $\kappa_V$.

\begin{lemma}[Eigenvalue gap, Proposition D.5 in \cite{banks2020pseudospectral}]
\label{lem:gapbound}
For any $M\in \bC^{n\times n}$ and any $t, \gamma >0$
$$\P\Big[\gap(M+\gamma G_n) \leq   t \Big] \leq \frac{n^3 t^2}{\gamma^2}.$$
\end{lemma}

\begin{lemma}[Eigenvector condition number ]
\label{lem:boundonkappaV}
For any $M\in \bC^{n\times n}, \gamma \in (0, \|M\|)$ and $t>0$ satisfying
$$  t < \frac{\gamma}{\|M\| n^{3/2}}, $$
we have 
$$\P\Big[\kappa_V(M+\gamma G_n)\geq \frac{ 1}{ t}  \Big] \leq 2    \left( 2\sqrt{2}+ \frac{\|M\|}{\gamma} +  \sqrt{\frac{4\log(1/t)}{n}}  \right)^2 n^3 t^2.$$
\end{lemma}

All of the ideas needed to prove Lemma \ref{lem:boundonkappaV} already appeared in \cite{banks2021gaussian}, but for the convenience of the reader  we quickly outline them below. First, we begin by recalling the following result. 

\begin{lemma}[Theorem 1.5 in \cite{banks2021gaussian}]
\label{lem:daviesexpectation}
Let $M\in \bC^{n\times n}$, $\gamma \in (0, \|M\|)$,  and let $\lambda_1, \dots, \lambda_n\in \bC$ be the random eigenvalues of $M+\gamma G_n$. Then for every measurable open set $B\subset \bC$
$$
    \E \bigg[ \sum_{\lambda_i \in B} \kappa(\lambda_i)^2 \bigg] \leq \frac{n^2}{\pi \gamma^2} \area (B). 
$$
\end{lemma}

We can now proceed to the proof. 

\begin{proof}[Proof of Lemma \ref{lem:boundonkappaV}]
To simplify notation put $X:= M+ \gamma G_n$ and let $\lambda_1, \dots, \lambda_n$ be its random eigenvalues. Then for any $s, t>0$
\begin{align*}
    \P\bigg[\kappa_V(X)\geq \frac{1}{t}\bigg] & = \P\bigg[\kappa_V(X)^2\geq  \frac{1}{t^2}\bigg]
    \\ & \leq \P\bigg[ \sum_{i=1}^n \kappa(\lambda_i)^2 \geq \frac{1}{n t^2}\bigg] && \kappa_V(X) \leq \sqrt{n\sum_{i=1}^n \kappa(\lambda_i)^2}
    \\ & \leq \P[\|G_n\|\geq s] + \P\bigg[\|G_n\| \leq s \text{ and } \sum_{i=1}^n \kappa(\lambda_i)^2 \geq  \frac{1}{n t^2} \bigg]. 
\end{align*}
   Moreover, from (\ref{eq:ginibrenormbound}) we have $\P[\|G_n\| \geq s]\leq 2\exp\big(-n(s-2\sqrt{2})^2\big)$. On the other hand 
\begin{align*}
 \P\left[\|G_n\| \leq s \text{ and } \sum_{i=1}^n \kappa(\lambda_i)^2 \geq \frac{1}{nt^2} \right] & \leq \P\left[ \sum_{\lambda_i\in D(0, \|M\|+s\gamma)} \kappa(\lambda_i)^2 \geq  \frac{1}{n t^2} \right]
 \\ & \leq \left(\frac{\|M\|}{\gamma}+s\right)^2 n^3t^2, 
\end{align*}
where the last inequality follows from Lemma \ref{lem:daviesexpectation} and Markov's inequality. Putting everything together we get that 
$$\P\bigg[\kappa_V(X)\geq \frac{1}{t}\bigg]\leq 2\exp\big(-n(s-2\sqrt{2})^2\big)+ \left(\frac{\|M\|}{\gamma}+s\right)^2 n^3t^2.$$
Now, to simplify notation define $\Rho := \frac{\|M\|}{\gamma} n^{3/2} t$. Then choose $s$ to be the solution of the equation $2\exp\big(-n(s-2\sqrt{2})^2\big) = \Rho^2$, and plug it into the above inequality to obtain
\begin{align*}
 \P\bigg[\kappa_V(X)\geq \frac{1}{t}\bigg]& \leq \Rho^2 +    \left( \frac{\|M\|}{\gamma}+ 2\sqrt{2} + \frac{1}{\sqrt{n}} \log(2/\Rho^2)  \right)^2 n^3 t^2
\\ &  \leq 2    \left( \frac{\|M\|}{\gamma}+ 2\sqrt{2} + \frac{1}{\sqrt{n}} \log(2/\Rho^2)  \right)^2 n^3 t^2
\\ &\leq 2    \left( \frac{\|M\|}{\gamma}+ 2\sqrt{2} + \frac{2}{\sqrt{n}} \log(1/t)  \right)^2 n^3 t^2 && 2 \Rho^{-2} \leq t^{-2}. 
\end{align*}
\end{proof}

We can now prove the shattering result.

\begin{proof}[Proof of Lemma \ref{lem:shattering}]
First, if we take $t_1:= \frac{\varphi^{1/2}\gamma}{\sqrt{2}n^{3/2}}$ and apply Lemma \ref{lem:gapbound} we get that 
$$\P[ \gap(M+\gamma G_n) \geq t_1 ] \geq 1-\varphi/2. $$
Then, taking $t_2 = \frac{\gamma \varphi^{1/2}}{60\|M\| \log(1/\varphi) n^{3/2}}$ and applying Lemma \ref{lem:boundonkappaV} we get 
\begin{align*}
\P[\kappa_V(M+\gamma G_n)\geq 1/t_2] & \leq  2    \left( 2\sqrt{2}+ \frac{\|M\|}{\gamma} + \frac{2}{\sqrt{n}} \log(1/t_2)^{1/2}  \right)^2 n^3 t^2_2
\\ &\leq 6 \left( 8 + \frac{\|M\|^2}{\gamma^2} + \frac{4}{n} \log(1/t_2) \right) n^3 t_2^2 &&  \text{AM-QM}
\\ & \leq \varphi/6+ \varphi/6+\varphi/6
\end{align*}
yielding 
$$\P[ \kappa_V(M+\gamma G_n) \leq 1/t_2 ] \geq 1-\varphi/2. $$

Now define $\shat = t_1/3$ and $\epsilon = t_1 t_2/3$. By the tail bounds obtained above we have the event $\{ \gap(M+\gamma G_n) \geq t_1 \text{ and } \kappa_V(M+\gamma G_n)\leq 1/t_2\}$ occurs with probability $1-\varphi$, and, by Lemma \ref{lem:pseudospectralbauerfike}, under this event we have that $\Lambda_\epsilon(M+\gamma G_n)$ is $\shat$-shattered, as we wanted to show. 
\end{proof}